\documentclass[11pt]{article} 

\usepackage{amsmath}
\usepackage{amssymb}
\usepackage[T1]{fontenc}
\usepackage[latin1]{inputenc}
\def\r{\rightarrow}

\newcommand{\fdem}{\hspace*{\fill}~$\Box$\par\endtrivlist\unskip}

\newcommand{\E}{\mathbb{E}}     
\renewcommand{\P}{\mathbb{P}}     
\renewcommand{\L}{\mathbb{L}}

\newcommand{\N}{\mathbb{N}}     

\newcommand{\R}{\mathbb{R}}     
     
\newcommand{\C}{\mathbb{C}}

\newcommand{\diag}{\mathop{\rm diag}}

\renewcommand{\r}{\mathop{\rightarrow}}

\newcommand{\cB}{\mbox{$\cal B$}}
\newcommand{\cC}{\mbox{$\cal C$}}
\newcommand{\cD}{\mbox{$\cal D$}}
\newcommand{\cE}{\mbox{$\cal E$}}
\newcommand{\cF}{\mbox{$\cal F$}}
\newcommand{\cG}{\mbox{$\cal G$}}
\newcommand{\cH}{\mbox{$\cal H$}}

\newcommand{\cL}{\mbox{$\cal L$}}
\newcommand{\cM}{\mbox{$\cal M$}}

\newcommand{\cR}{\mbox{$\cal R$}}

\newtheorem{theo}{Theorem}
\newtheorem{pro}{Proposition}
\newenvironment{proof}[1]{{\it Proof. }#1}{\fdem}
\newtheorem{lem}{Lemma}
\newtheorem{rem}{Remark}
\newtheorem{cor}{Corollary}
\begin{document}
\begin{center} 
{\large\bf A RENEWAL THEOREM 
FOR STRONGLY ERGODIC MARKOV \\[0.15cm]
CHAINS IN DIMENSION $d\geq3$ AND CENTERED CASE.}  
\end{center}
\vskip 3mm
\begin{center}
D. GUIBOURG and L. HERV\'E \footnote{Universit\'e
Europ\'eenne de Bretagne, I.R.M.A.R. (UMR-CNRS 6625), Institut National des Sciences 
Appliqu\'ees de Rennes. Denis.Guibourg@ens.insa-rennes.fr, Loic.Herve@insa-rennes.fr} 
\end{center}



\vskip 3mm

\noindent{\bf Abstract.} {\scriptsize \it In dimension $d\geq3$, we present a general assumption under which the renewal theorem established by Spitzer \cite{spitzer} for i.i.d.~sequences of centered nonlattice r.v.~holds true. Next we appeal to an operator-type procedure to investigate the Markov case. Such a spectral approach has been already developed by Babillot  \cite{bab}, but the weak perturbation theorem of \cite{keli} enables us to greatly weaken the moment conditions of \cite{bab}. Our applications concern the $v$-geometrically ergodic Markov chains, the $\rho$-mixing Markov chains, and the iterative  Lipschitz models, for which the renewal theorem of the i.i.d.~case extends under the (almost) expected moment condition. }

\vskip 3mm

\section{Introduction} \label{intro}
Let $(\Omega,\cF,\P)$ be a probability space, let $(E,\cE)$ be a measurable space, and let $(X_n,S_n)_{n\geq0}$ be a  sequence of random variables (r.v.) defined on $\Omega$ and taking values in $E\times\R^d$. We denote by $\cL_d(\cdot)$ the Lebesgue-measure on $\R^d$, by $\langle\cdot,\cdot\rangle$ the canonical scalar product on $\R^d$, by $\|\cdot\|$ the euclidean norm in $\R^d$, and by $B(\R^d)$ the Borel $\sigma$-algebra on $\R^d$, and we use the notation ''$\, ^*\, $'' to denote the matrix transposition. 

{\it Throughout this paper we assume that $d\geq3$}. Given some fixed nonnegative measurable function $f$ on $E$, we are interested in the asymptotic behavior of the renewal-type measures 
%on $\R^d$ defined by $A\in B(\R^d)\mapsto
$\sum_{n=1}^{+\infty}\E[f(X_n)\, 1_A(S_n-a)]\, $ ($A\in B(\R^d)$) when $a\in\R^d$ goes to infinity. 
%$\|a\|\r+\infty\, $ ($a\in\R^d$), . 
Specifically we will consider the centered case in a sense that will be specified later. In the i.i.d.~case, this corresponds to the well-known result established by Spitzer \cite{spitzer}: if $(X_n)_{n\geq1}$ is a  i.i.d.~sequence of centered nonlattice r.v.~such that $\E[\, |X_0|^{m_d}\, ] < \infty$ with $m_d = \max\{d-2,2\}$, and if the covariance $d\times d$-matrix $\Sigma = \E[X_0\, X_0^*]$ is invertible, then the associated random walk $S_n = \sum_{k=1}^{n} X_k$ satisfies the following property when $\|a\|\r+\infty$:  
\begin{equation} \label{ren-intro}
\forall g\in\cC_c(\R^d),\ \ \ \ \sum_{n=1}^{+\infty}\E_\mu\big[\, g(S_n-a)\, \big] \sim \frac{C_d}{\langle \Sigma^{-1}a,a\rangle^{\frac{d-2}{2}}}\, \cL_d(g)
\end{equation}
with 
\begin{equation} \label{C_d} 
C_d = 2^{-1}\, \pi^{-\frac{d}{2}}\, (\det\Sigma)^{-\frac{1}{2}}\, \Gamma\big(\frac{d-2}{2}\big), 
\end{equation}
where $\Gamma(\cdot)$ denotes the usual Gamma-function and $\cC_c(\R^d)$ stands for the space of complex-valued compactly supported continuous functions on $\R^d$. 
%The purpose of this paper is double. 

In Section~\ref{vers-gene} we present some general assumptions on the characteristic-type functions $\E[f(X_n)\, e^{i\langle t, S_n\rangle}]$ under which the previous renewal theorem  extends. In fact these assumptions are the ``tailor-made'' conditions for proving a multidimensional renewal theorem in the centered case with the help of  the Fourier techniques. This part extends and specifies (and sometimes simplifies) some of the arguments and computations used in  Babillot's paper \cite{bab}.  

The purpose of Section~\ref{markov} is to investigate the general assumptions of Section~\ref{vers-gene} in the case when $(X_n)_{n\geq0}$  is a Markov chain. In this context we show in Subsection~\ref{subsect_semi_group} that a natural way 
%for investigating the above mentioned characteristic-type functions 
is to assume that, for each $t\in\R^d$, the linear operators $Q_n(t)$ defined by $(Q_n(t)f)(x) : = \E_{x}[f(X_n)\, e^{i\langle t, S_n\rangle}]\, $ ($x\in E$) 
is a semi-group: $Q_{m+n}(t) = Q_n(t)\circ Q_m(t)$. 
The Markov random walks 
%and the additive functionals of a Markov chain 
satisfy this assumption. We then have $\E_x[f(X_n)\, e^{i\langle t, S_n\rangle}] = (Q_1(t)^nf)(x)$, and if $Q_1(0)$ satisfies some ''good'' spectral properties on a Banach space $\cB$, then the operator-perturbation method first introduced by Nagaev \cite{nag1} can be used to obtain an accurate expansion of $Q_1(t)^nf$ for any $f\in\cB$, thus of $\E_x[f(X_n)\, e^{i\langle t, S_n\rangle}]$. This spectral method has been already exploited to investigate the Markov renewal theorems, see e.g.~\cite{bab,guihar,broi,hulo,fuhlai}. In particular in \cite{bab}, Property~(\ref{ren-intro}) is extended to centered Markov random walks in dimension $d\geq 3$. However, because of the use of the standard perturbation theory, the assumptions in \cite{bab} (and in all the previously cited works) need some operator-moment conditions which are restrictive, even not fulfilled in general when $S_1$ is unbounded.  

In this work, we use the weak Nagaev method introduced in \cite{aap,ihp1,bal-seb,gui-lepage}, which allows to greatly weaken the moment conditions in limit theorems for strongly ergodic Markov chains and dynamical systems. This new approach, in which the Keller-Liverani perturbation theorem \cite{keli,liverani,gouliv} plays a central role, is fully described in \cite{fl} and applied to prove some usual refinements of the central limit theorem (CLT). It is used in \cite{denis} to establish a one-dimensional (non-centered) Markov renewal theorem. Also mention that one of the more beautiful applications of this method concerns the convergence to stable laws, see \cite{gui-lepage,BDG,seb-08}. 

This new approach is outlined in Subsection~\ref{nag_method} and applied in Section~\ref{applic} to the three following classical strongly ergodic Markov models:  \\[0.15cm] {\it 
\indent - The $v$-geometrically ergodic Markov chains \cite{mey}, \\[0.1cm]
\indent - The $\rho$-mixing Markov chains \cite{rosen}, \\[0.1cm]
\indent - The iterative Lipschitz models \cite{duf}. }\\[0.15cm]
To have a good understanding of our improvements in term of moment conditions, let us consider an unbounded function $v : E\r[1,+\infty)$ and a $v$-geometrically ergodic Markov chain $(X_n)_{n\geq0}$, that is we have: $\sup_{|f|\leq v}\sup_{x\in E}|Q^nf(x)-\pi(f)|/v(x) = O(\kappa^n)$ for some $\kappa\in(0,1)$, 
where $Q$ denotes the transition probability of the chain and $\pi$ is the stationary distribution. Besides, let us consider the 
Markov random walk $S_n = \sum_{k=1}^{n} \xi(X_k)$ in $\R^d$ associated to some measurable function $\xi : E\r\R^d$.  
The centered case corresponds to the assumption $\pi(\xi) = 0$, the (asymptotic) covariance matrix is given by 
$\Sigma \ : = \lim_n \E[S_nS_n^*]/n$, 
provided that this limit exists, and finally the (Markov) nonlattice condition means that there exist no $b\in\R^d$, no closed subgroup $H$ in $\R^d$, $H\neq \R^d$, no $\pi$-full $Q$-absorbing set $A\in\cE$, and finally no bounded measurable function $\theta\, :\, E\r\R^d$ such that:  
$\forall x\in A, \ \ \xi(y)+\theta(y)-\theta(x)\in b+H\ \ Q(x,dy)-$p.s. We shall prove in Subsection~\ref{sub_sec_v_geo} the next statement: 

\noindent {\it Let us assume that $(X_n)_{n\geq0}$ is a $v$-geometrically ergodic Markov chain, that $\xi : E\r\R^d$ ($d\geq 3$) is centered, nonlattice, and such that $\|\xi\|^{m_d+\delta_0}\leq C\, v$ with $m_d = \max\{d-2,2\}$ and some constants $C, \delta_0>0$, and finally that the initial distribution $\mu$ of the chain is such that $\mu(v)<\infty$. Then the above asymptotic covariance matrix $\Sigma$ is well-defined and positive definite, and $S_n = \sum_{k=1}^{n} \xi(X_k)$ satisfies (\ref{ren-intro}). } 

\noindent In comparison with the optimal order $m_d$ of the i.i.d.~case, the condition $\|\xi\|^{m_d+\delta_0}\leq C\, v$ is the expected moment assumption (up to $\delta_0>0$) in the context of $v$-geometrically ergodic Markov chains. 
%(for instance, $\|\xi\|^2 \leq C\, v$ implies the CLT). 
The corresponding result in \cite{bab} requires the following assumption: $\sup_{x\in E} \frac{1}{v(x)}\int_E |\xi(y)|^{m_d+\delta_0}\, v(y)\, Q(x,dy) < \infty$, see also \cite{fuhlai}. This moment condition is not only clearly stronger than the previous one, but actually it is not fulfilled in general when $\xi$ is unbounded. A typical example is presented in \cite[$\S$~3]{fl} for the usual linear autoregressive model and $\xi(x) = x$.  Similar improvements concerning the moment conditions are obtained in Subsections~\ref{sub_sec_L2} and \ref{sub_sec_ite} for the two others (above cited) Markov models. 

\noindent Throughout this paper we shall use the following definitions. Let ${\cal O}$ be an open subset of $\R^d$, let $(X,\|\cdot\|_X)$ be a normed vector space, and let $\tau\in[0,1)$. We say that a function $V : {\cal O}\r X$ is uniformly $\tau$-holderian if there exists a constant $C\geq0$ such that we have: 
$$\forall (t,t')\in {\cal O}^2,\ \|V(t)-V(t')\|_X \leq C\, \|t-t'\|^\tau.$$
If $m\in\R_+$, we shall say that $U\in\cC_b^m({\cal O},X)$ if $U$ is a function from ${\cal O}$ to $X$ satisfying the following properties~:  \\[0.12cm]
\indent $U$ is $\lfloor m\rfloor$-times continuously differentiable on ${\cal O}$ \\[0.12cm] 
\indent Each partial derivative of order $j=0,\ldots,\lfloor m\rfloor$ of $U$ is bounded on ${\cal O}$ \\[0.12cm]
\indent Each partial derivative of order $\lfloor m\rfloor$ of $U$ is uniformly $(m-\lfloor m\rfloor)$-hölder on ${\cal O}$.  
%=====================
\section{A general statement} \label{vers-gene}
In this section, we state a general result that will be applied later to the Markov context, but which has its own interest in view of other possible applications. Here $(X_n,S_n)_{n\geq1}$ is a general sequence of random variables defined on $\Omega$ and taking values in $E\times\R^d$. 

\noindent For any $R>0$, we set $B(0,R) := \{t\in\R^d : \|t\| < R\}$, and for any $0<r<r'$, we set $K_{r,r'} = \{t\in\R^d : r < \|t\| < r'\}$. We denote by $\nabla$ and $Hess$ the gradient and  the Hessian matrix, respectively. 

\noindent{\bf Hypothesis $\cR(m)$.}  {\it Given a real number $m>0$ and a measurable function $f : E\r[0,+\infty)$ such that $\E[f(X_n)] < \infty$ for all $n\geq 1$, 
we will say that Hypothesis $\cR(m)$ holds if: \\[0.12cm]
(i) There exists $R>0$ such that we have for all $t\in B(0,R)$ and all $n\geq 1$: 
\begin{equation} \label{X_n-S-n}
\E\big[f(X_n)\, e^{i\langle t, S_n\rangle}\big] = \lambda(t)^n\, L(t) + R_n(t), 
\end{equation}
where $\lambda(0)=1$, the functions $\lambda(\cdot)$ and $L(\cdot)$ are in $C_b^m\big(B(0,R),\C\big)$, and the series $\sum_{n\geq1}R_n(\cdot)$ uniformly converges on $B(0,R)$ and defines a function in $ C_b^m\big(B(0,R),\C\big)$. \\[0.12cm]
(ii) For all $0<r<r'$, the series $\sum_{n\geq1}\E\big[f(X_n)\, e^{i\langle \cdot, S_n\rangle}\big]$ uniformly converges on $K_{r,r'}$ and  defines a function in $ C_b^m\big(K_{r,r'},\C\big)$. } 

\noindent If Hypothesis $\cR(m)$ holds with $m\geq 2$, we define the following  symmetric matrix: 
$$\Sigma : = - Hess\, \lambda(0).$$ 
\noindent  In the sequel, $\vec m := \frac{\nabla\lambda(0)}{i}$ will be viewed as an asymptotic first moment vector in $\R^d$, and we shall assume that $\vec m = 0$. Under this centered assumption, $\Sigma= - Hess\, \lambda(0)$ may be viewed as an asymptotic covariance matrix. These facts are specified in the following proposition, in which we use the notation $F^{(\ell)}$ for the derivative of order $\ell$ of a complex-valued function $F$.   
\begin{pro} \label{deri-hess} $\ $ \\[0.1cm]
(i) If Condition~$\cR(1)$-(i) is fulfilled with $f=1_E$, $L(0)=1$, $\sup_{n\geq1}|R_n^{(1)}(0)| < \infty$, and if $\forall n\geq1,\ \E[\|S_n\|] < \infty$, then we have: $\displaystyle\frac{\nabla\lambda(0)}{i} = \lim_n\frac{1}{n}\, \E[S_n]$.  \\[0.15cm]
(ii) If Condition~$\cR(2)$-(i) is fulfilled with $f=1_E$, $L(0)=1$, $\sup_{n\geq1}|R_n^{(2)}(0)| < \infty$, and if in addition we have $\nabla\lambda(0) = 0$ and $\forall n\geq1,\ \E[\|S_n\|^2] < \infty$, then we have: $\displaystyle \Sigma = \lim_n\frac{1}{n}\, \E[S_nS_n^*]$. 
\end{pro}
\begin{proof} To simplify the computations, let us assume $d=1$. The extension to $d\geq2$ is obvious by using partial derivatives. By deriving  $\E[e^{itS_n}] = \lambda(t)^n\, L(t) + R_n(t)$ at $t=0$, we obtain: $i\, \E[S_n] = n\, \lambda^{(1)}(0) + L^{(1)}(0) + R_n^{(1)}(0)$. Hence Assertion~(i).  
Next by deriving twice the same equality  at $t=0$, one gets: $- \, \E[S_n^2] = n\, \lambda^{(2)}(0) + L^{(2)}(0) + R_n^{(2)}(0)$. Hence Assertion~(ii). \end{proof}
Notice that the assumption $\sup_{n\geq1}|R_n^{(\ell)}(0)| < \infty$ used in Proposition~\ref{deri-hess} is natural in view of the desired regularity properties of $\sum_{n\geq1}R_n(\cdot)$ in Hypothesis~$\cR(m)$-(i). 

\noindent For $k\in\N^*$, we denote by $\cH_k$ the space of all complex-valued continuous Lebesgue-integrable functions on $\R^d$, whose Fourier transform is compactly supported and $k$-times continuously differentiable on $\R^d$. 
The constant $C_d$ is defined in (\ref{C_d}).  
\begin{theo} \label{ren-theo} 
Assume that $d\geq 3$, that Hypothesis~$\cR(m)$ holds with $m=\max\{d-2,2+\varepsilon\}$ for some $\varepsilon>0$, that  $\nabla\lambda(0) = 0$, and that $\Sigma$ is positive definite. Then we have for all $g\in\cH_{d-2}$: 
\begin{equation} \label{ren-rate-theo}
\sum_{n=1}^{+\infty}\E\big[\, f(X_n)\, g(S_n-a)\, \big] \sim 
\frac{C_d\, L(0)}{\langle \Sigma^{-1}a,a\rangle^{\frac{d-2}{2}}}\, \cL_d(g)\ \ \mbox{when}\ \  \|a\|\r+\infty\ \ (a\in\R^d).
\end{equation}
\end{theo}
%\noindent The following proposition shows that, in practice\footnote{As seen in Section~\ref{applic}, the assumption $\sup_{n\geq1}|R_n^{(\ell)}(0)| < \infty$ used in Proposition~\ref{deri-hess} is natural in view of the desired regularity properties of $\sum_{n\geq1}R_n(\cdot)$ in Hypothesis~$\cR(m)$-(i).}, the above assumption $\nabla\lambda(0) = 0$ follows from a centered condition, while $\Sigma$ may be viewed as an asymptotic covariance matrix. 
%\begin{pro} \label{deri-hess}
%Let us assume that $\E[\|S_n\|^2] < \infty$, $\E[S_n] = 0$,  that Hypothesis~$\cR(m)$-(i) is fulfilled with $m\geq2$, $f=1_E$, $L(0)=1$, and 
%is reinforced by the following assumptions\footnote{}: 
%$\sup_{n\geq1}|R_n^{(\ell)}(0)| < \infty$ for $\ell= 1,2$.  
%\\[0.08cm] 
%Then we have $\nabla\lambda(0) = 0$ and $\displaystyle \Sigma = \lim_n\frac{1}{n}\, \E[S_nS_n^*]$. 
%\end{pro}
%\noindent{\it Proof of Proposition~\ref{deri-hess}.} To simplify the derivative computations, let us assume $d=1$. The extension to $d\geq2$ is obvious by using partial derivatives. By deriving  $\E[e^{itS_n}] = \lambda(t)^n\, L(t) + R_n(t)$ at $t=0$, we obtain: $0 = i\, \E[S_n] = n\, \lambda'(0) + L'(0) + R_n'(0)$. Hence $\lambda'(0)=0$. Next by deriving twice the same equality  at $t=0$, one gets: $- \, \E[S_n^2] = n\, \lambda''(0) + L''(0) + R_n''(0)$. Hence $\lambda''(0) = - \lim_n\frac{1}{n}\, \E[S_n^2]$. \fdem

\noindent It is well-known (see e.g.~\cite{bre}) that the conclusion of Theorem~\ref{ren-theo} implies that, for all $a\in\R^d$, $U_a(B) := \sum_{n=1}^{+\infty}\E[\, f(X_n)\, 1_B(S_n-a)\, ]\, $ ($B\in B(\R^d)$) 
defines a positive Radon measure on $\R^d$, and that $\langle \Sigma^{-1}a,a\rangle^{d-2/2}\, U_a$ weakly converges to $C_d\, L(0)\, \cL_d(\cdot)$ when $\|a\|\r+\infty$. Consequently: 
\begin{cor} \label{ren-coro}
Under the assumptions  of Theorem~\ref{ren-theo}, the property (\ref{ren-rate-theo}) is fulfilled with any real-valued  continuous compactly supported function $g$ on $\R^d$, and also with $g=1_B$ for any bounded Borel set $B$ in $\R^d$ 
whose boundary has a zero Lebesgue-measure. 
\end{cor}
The proof of Theorem~\ref{ren-theo} is based on some Fourier techniques partially derived from \cite{bab}. In the case $d=3$ or $4$, the optimal order 2 of the i.i.d.~case (see Section~\ref{intro}) is here replaced with the order $2+\varepsilon$. This will be needed in the proof of Lemma~\ref{I1(a)-I3(a)} below. 

\noindent {\it Proof of Theorem~\ref{ren-theo}.} 
Set $E_n(t) := \E[f(X_n)\, e^{i\langle t, S_n\rangle}]$ for each $t\in\R^d$. Let $h\in\cH_{d-2}$, and let $b>0$ be such that $\hat h(t)=0$ if 
$\|t\|>b$. The inverse Fourier formula on $h$ easily gives \\[0.15cm]
\indent $\displaystyle \ \ \ \ \ \ \ \ \ \ \ \ \ \ \forall n\in\N^*, \ \ 
(2\pi)^d\, \E[f(X_n)\, h(S_n)] = \int_{\|t\|\leq b} \hat h(t)\,E_n(t)\, dt$. \\[0.15cm]
%Let $|\cdot|$ denote the uniform norm on $\R^d$, and 
Since $\lambda(t) = 1 -\frac{1}{2}\langle \Sigma t,t\rangle + o(\|t\|^2)$, one can choose $\alpha< R$ such that  
\begin{equation} \label{lambda(t)-inegalite}
\|t\| \leq \alpha\ \Rightarrow\ |\lambda(t)| \leq 1 -\frac{1}{4}\langle \Sigma t,t\rangle.
\end{equation}
Now let $r\in(0,\alpha)$, and let $\chi$ be a real-valued compactly supported and indefinitely differentiable function on $\R^d$, such that its support is contained in $\{t\in\R^d : \|t\|\leq \alpha\}$ and $\chi(t)=1$ for $\|t\|\leq r$. Let us write \\[0.15cm]
\noindent $\displaystyle (2\pi)^d\, \E\big[f(X_n)\, h(S_n)\big] = \int_{\|t\| \leq \alpha} \chi(t)\hat h(t)\, E_n(t)\, dt + 
\int_{r<\|t\| \leq b} (1-\chi(t))\, \hat h(t)\, E_n(t)\, dt$. \\[0.15cm]
Let $a\in\R^d$ and set $h_a(\cdot) := h(\cdot -a)$. We have $\widehat{h_a}(t) = \hat h(t)\, e^{-i\langle t,a \rangle}$. 
By applying the previous equalities to $h_a$ and using Hypothesis $\cR(m)$, one obtains  \\[0.12cm]
\indent $\displaystyle  \ \ \ \ \ \ \ \ \ \ \ \ \ \ \ \ \ \ \ 
(2\pi)^d\, \sum_{n=1}^{+\infty}\E_\mu[f(X_n)\, h(S_n-a)] = I(a) + J(a) + K(a)$ 
\begin{eqnarray*}
\mbox{with}\ \ \ I(a) &=& 
\int_{\|t\| \leq \alpha} \chi(t)\, \hat h(t)\, \frac{\lambda(t)}{1-\lambda(t)} \, L(t)\, e^{-i\langle t,a\rangle}\, dt \\
 J(a) &=& 
\int_{\|t\| \leq \alpha} \chi(t)\, \hat h(t)\, \big(\sum_{n\geq1} R_n(t)\big) \, e^{-i\langle t,a\rangle}\, dt \\
K(a) &=& 
\int_{r<\|t\| \leq b} (1-\chi(t))\, \hat h(t)\, \big(\sum_{n\geq1} E_n(t)\big)\, e^{-i\langle t,a\rangle}\, dt.  
\end{eqnarray*}
Indeed, it follows from (\ref{lambda(t)-inegalite}) that $\forall t\in B(0,\alpha)$, $t\neq 0$, we have $|\lambda(t)| <1$ and  $\sum_{n\geq1}|\lambda(t)|^n = \frac{|\lambda(t)|}{1-|\lambda(t)|} \leq \frac{4}{\langle \Sigma t,t\rangle}$. Since $d\geq3$ and $\Sigma$ is invertible, the function $t\mapsto \frac{1}{\langle \Sigma t,t\rangle}$ is integrable near $t=0$. 
So the term $I(a)$ derives from Lebesgue's theorem. The terms $J(a)$ and $K(a)$ follow from the uniform convergence of the series $\sum_{n\geq1}R_n(\cdot)$ and $\sum_{n\geq1}E_n(\cdot)$ on $B(0,\alpha)$ and $K_{r,b}$ respectively. Theorem~\ref{ren-theo} is then a consequence of the three next lemmas. \fdem
\begin{rem} 
The uniform convergence stated in  Hypothesis~$\cR(m)$ for the series $\sum_{n\geq1}R_n(\cdot)$ and $\sum_{n\geq1}E_n(\cdot)$ is used for defining $J(a)$ and $K(a)$. Of course,  alternative conditions may be used, 
%in  Hypothesis~$\cR(m)$, 
as for instance: $\sum_{n\geq1}\int_{B(0,R)}|R_n(t)|dt < \infty$ and $\sum_{n\geq1}\int_{K_{r,r'}} |E_n(t)|dt < \infty$. 
\end{rem}
\begin{lem} \label{J(a)-K(a)}
We have $J(a) + K(a) = o(\|a\|^{-(d-2)})$ when $\|a\| \r+\infty$. 
%Furthermore, if $m\in\N$, then we have  $J(a) + K(a) = o(\|a\|^{-m})$. 
\end{lem}
\begin{proof} By Hypothesis~$\cR(m)$, the integrands in $J(a)$ and $K(a)$ are respectively in $\cC_b^m(B(0,R),\C)$ (with $R > \alpha$) and in $\cC_b^m(K_{r/2,2b},\C)$ . Since $m=\max\{d-2,2+\varepsilon\}\geq d-2$, this gives the desired result. 
\end{proof}
Now let us investigate $I(a)$. An easy computation yields $I(a) = I_1(a) + I_2(a) +I_3(a)$ with 
\begin{eqnarray*}
I_1(a) &:=& 
\int_{\|t\| \leq \alpha} \chi(t)\, \frac{\hat h(t)\lambda(t)\, L(t) - \hat h(0)L(0)}{1-\lambda(t)}\, 
e^{-i\langle t,a\rangle}\, dt \\
I_2(a) &:=& 
2\, \hat h(0)\, L(0)\, \int_{\|t\| \leq \alpha} \frac{\chi(t)}{\langle \Sigma t,t\rangle}\, e^{-i\langle t,a\rangle}\, dt \\
I_3(a) &:=& 
2\, \hat h(0)\, L(0)\, \int_{\|t\| \leq \alpha} \chi(t)\, \frac{\lambda(t) - 1 + \frac{1}{2} \langle \Sigma t,t\rangle}
{(1-\lambda(t))\langle \Sigma t,t\rangle}\, e^{-i\langle t,a\rangle}\, dt. 
\end{eqnarray*}
\begin{lem} \label{I1(a)-I3(a)} 
We have  $I_1(a) + I_3(a) = o(\|a\|^{-(d-2)})$ when $\|a\| \r+\infty$. 
\end{lem}
\begin{proof}
%Given a complex-valued function $F$ defined on an open subset of $\R^d$, 
%we shall simply employ the notation $F^{(\ell)}$ to denote any partial derivative of order $\ell$ of $F$. \\
To study $I_1(a)$, let us define the following function on $\R^d$~: 
$$\theta_1(t) = \chi(t)\left(\hat h(t)\lambda(t)\, L(t) - \hat h(0)L(0)\right).$$
Then  $\theta_1$ has $d-2$ continuous derivatives on $\R^d$ and since $\theta_1(0)=0$, 
we have $|\theta_1(t)| = O(\|t\|)$ on $\R^d$. Now set $u(t) = \frac{\theta_1(t)}{1-\lambda(t)}$.  
Some standard derivative arguments  give (see Remark~(a) in Appendix~A.0)
\begin{equation} \label{deri-u}
\forall i=0,\ldots,d-2,\ \ |u^{(i)}(t)| = O(\|t\|^{-(1+i)}),  
\end{equation}
where $u^{(i)}(\cdot)$ stands for any partial derivative of order $i$ of $u$. 
In particular the partial derivatives of order $d-2$ of $u$ are Lebesgue-integrable on $\R^d$.  Since $I_1(a) = \hat u(a)$, 
the claimed property for  $I_1(a)$ follows from Proposition~A.1 in Appendix A.1.  
The same property can be similarly established for $I_3(a)$ by defining the function 
$$\theta_3(t) = \chi(t)\, \left(\lambda(t) - 1 + \frac{1}{2} \langle \Sigma t,t\rangle\right).$$ 
Indeed, if $d\geq 5$, then $\theta_3$ has three continuous derivatives on $\R^d$ and it can be easily seen that 
$|\theta_3(t)| = O(\|t\|^3)$, $|\theta_3^{(1)}(t)|= O(\|t\|^2)$  and $|\theta_3^{(2)}(t)|= O(\|t\|)$.  
Let $v(t) = \frac{\theta_3(t)}{(1-\lambda(t))\langle \Sigma t,t\rangle}$. 
Remark~(b) in Appendix~A.0 then yields 
\begin{equation} \label{deri-v}
\forall i=0,\ldots,d-2,\ \ |v^{(i)}(t)| = O(\|t\|^{-(1+i)}). 
\end{equation}
If $d=3$ or 4, then we only have $|\theta_3(t)| = O(\|t\|^{2+\varepsilon})$, $|\theta_3^{(1)}(t)|= O(\|t\|^{1+\varepsilon})$  and $|\theta_3^{(2)}(t)|= O(\|t\|^\varepsilon)$. Similarly, one gets 
\begin{equation} \label{deri-v'}
\forall i=0,\ldots,d-2,\ \ |v^{(i)}(t)| = O(\|t\|^{-(i+2-\varepsilon)}). 
\end{equation}
From either (\ref{deri-v}) or (\ref{deri-v'}), the desired estimate on $I_3(a) = 2\, \hat h(0)\, \, L(0)\, \hat v(a)$ again follows from Proposition~A.1.  
\end{proof}
\begin{lem} \label{I2(a)}
Set $\displaystyle C'_d = (2\pi)^{\frac{d}{2}}\, 2^{\frac{d}{2}-1}\, 
\Gamma\left(\frac{d-2}{2}\right) (\det\Sigma)^{-\frac{1}{2}}$. We have 
$\displaystyle I_2(a) \sim \frac{C'_d\, \hat h(0)\, L(0)}{\langle \Sigma^{-1}a,a\rangle^{\frac{d-2}{2}}}$ when $\|a\| \r+\infty$.  
\end{lem}
\begin{proof}
Let $S(\R^d)$ denote the so-called Schwartz space. Since $\chi\in S(\R^d)$ and the Fourier transform is a bijection on $S(\R^d)$, let us call $\psi\in S(\R^d)$ such that $\widehat{\psi}=\chi$. Then, setting $\psi_a(\cdot) := \psi(\cdot -a)$, we have 
$$I_2(a) = 2\hat{h}(0)\, L(0)\,  \int_{\R^d}\frac{\widehat{\psi_a}(t)}{\langle \Sigma t,t\rangle}dt.$$
Set $\Delta=\diag(\sqrt{\lambda_1},\ldots,\sqrt{\lambda_d})$, where 
the $\lambda _i$'s are the eigenvalues of the covariance matrix $\Sigma$ ($\lambda _i>0$ because  $\Sigma$ is invertible), 
and let $P$ be any orthogonal matrix of order $d$ such that 
$$P^{-1}\Sigma P = \Delta^2 = \diag(\lambda_1,\ldots,\lambda_d).$$
We have $\langle \Sigma t,t\rangle = \langle \Delta^2 P^{-1}t ,  P^{-1}t\rangle = \|\Delta  P^{-1}t\|^2$. So, by 
setting $t=P\Delta^{-1}u$, one gets 
$$I_2(a) = 2\, \hat{h}(0)\, L(0)\, \int (\det\Delta)^{-1}\, \widehat{\psi_a}(P\Delta^{-1}u) \, \frac{1}{\|u\|^2}du.$$
The Fourier transform of $g(v) = \psi(P\Delta v-a)$ is 
$\hat g(u) = (\det\Delta)^{-1}\, \widehat{\psi_a}(P\Delta^{-1}u)$. So we have 
\begin{equation} \label{dist}
I_2(a) = 2\, \hat{h}(0)\, L(0) \, c\, \int \psi(P\Delta v-a)\, \frac{1}{\|v\|^{d-2}}\,  dv,
\end{equation}
where $c=(2\pi)^{\frac{d}{2}}\, 2^{\frac{d}{2}-2}\, \Gamma(\frac{d-2}{2})$. Indeed the equality (\ref{dist}) follows from the fact that we have  $\widehat{(\|\cdot\|^{-2})}(v) = c\, \|v\|^{2-d}$ in the sense of temperated distribution  on $\R^d$. (See also Appendix~A.2 for an elementary proof of (\ref{dist})). Now set $b = \Delta^{-1}\, P^{-1}\, a$,  $F(x) = \psi(-P\Delta\, x)$. 
Furthermore, set $\beta=\|b\|$, $b=\beta\, \tilde b$, so $\|\tilde b\| = 1$, and define $F_{\beta}(x) = \beta^d\, F(\beta\, x)$ for any $\beta>0$. Then 
$$I_2(a) = 2\, \hat{h}(0)\, L(0) \, c\, \int F(b-v) \, \frac{1}{\|v\|^{d-2}}\, dv \ = \  
2\, \hat{h}(0)\, L(0) \, c\, \int F_\beta(\tilde b - w) \frac{1}{\|\beta w\|^{d-2}} \, dw,$$
so that $\beta^{d-2}\, I_2(a) = 2\, c\, \hat{h}(0)\, L(0)\, (F_\beta* f_{d-2})(\tilde b)$, where $f_{d-2}(w) := \frac{1}{\|w\|^{d-2}}$, and $*$ denotes the convolution product on $\R^d$. By using some standard arguments of approximate identity, one can prove (see Appendix A.3.) that the following convergence holds uniformly in 
$\tilde b\in\R^d$ such that $\|\tilde b\| =1$~: 
\begin{equation} \label{id-approchee}
\lim_{\beta\r+\infty}(F_{\beta}* f_{d-2})(\tilde b) = \int_{\R^d} F(w)dw.
\end{equation}
\noindent We have $\int F(w)dw =  (\det\Delta)^{-1}\int \psi(y)dy = (\det\Sigma)^{-\frac{1}{2}}$, because $\int \psi(y)dy = \hat{\psi}(0) = \chi(0) = 1$, and $\beta= \|Pb\| = \|\Sigma^{-\frac{1}{2}}a\| = \langle \Sigma^{-1}a,a\rangle^{\frac{1}{2}}$. The above computations then yield Lemma \ref{I2(a)}. 
\end{proof}
%=================================
\section{Operator-type procedure in Markov models} \label{markov}
%==================================
In this section $(X_n,S_n)_{n\geq0}$ still denotes a  sequence of random variables taking values in $E\times\R^d$, but from now on $(X_n)_{n\geq 0}$ is assumed to be a Markov chain with state space $E$, transition probability $Q(x,dy)$, and initial distribution $\mu$. We suppose that $S_0=0$. The functional action of $Q(x,dy)$ is given by $Qf(x) = \int_E f(y)\, Q(x,dy)$ provided that this integral is well defined. In this Markov context, we first present a general assumption providing an operator-type formula for the term $\E[f(X_n)\, e^{i\langle t, S_n\rangle}]$ of Hypothesis~$\cR(m)$. Next, we briefly compare the spectral method developed in \cite{bab} with that presented in \cite{fl}. 
%, which will be applied in the classical Markov models of Section~\ref{applic}. 
%=========================================
\subsection{A Markov context} \label{subsect_semi_group}
%We denote by $\L^2(\pi)$ the usual Lebesgue space associated to $\pi$. 
%we denote by  $\P_\mu$ the underlying probability measure (and by $\E_\mu$ the associated expectation) to refer to the initial distribution $\mu$ of $(X_n)_{n\geq 0}$, with the special notation $\P_x$ (and $\E_x$) when $\mu$ is the Dirac distribution at some $x\in E$. 
For the moment we make the following abuse of notation: given $x\in E$, we denote by  $\P_x$ the underlying probability measure (and by $\E_x$ the associated expectation) to refer to the case when the initial distribution of $(X_n)_{n\geq 0}$, is the Dirac distribution at $x$. This notation takes a precise (and usual) sense in the example below and in the setting of Section~\ref{applic}. 
For any bounded measurable function $f$ on $E$, and for $n\in\N$, $t\in\R^d$, let us define: 
\begin{equation} \label{Qn(t)}
\big(Q_n(t)f\big)(x) := \E_x\big[e^{i\, \langle t, S_n\rangle}\, f(X_n)\big]\ \ \ \ (x\in E).
\end{equation}
Let us observe that  $Q_1(0) = Q$. Let us consider the following condition: 

\noindent {\bf Condition ($\cG$).} {\it For all $t\in\R^d$, $(Q_n(t)_{n\in\N}$ is a semi-group, that is:  }
%\begin{equation} \label{semi-group}
$$\forall(m,n)\in\N^2,\ \ Q_{m+n}(t) = Q_m(t) \circ Q_n(t).$$
%\end{equation}
Under Condition~($\cG$), we have in particular $Q_{n}(t) = Q(t)^n$ where $Q(t):=Q_1(t)$ is defined by: 
\begin{equation} \label{noyau-fourier}
\big(Q(t)f\big)(x) := \E_x\big[e^{i\, \langle t, S_1\rangle}\, f(X_1)\big]\ \ \ \ (x\in E). 
\end{equation}
The $Q(t)$'s are called the Fourier operators, and in view of the study of Hypothesis~$\cR(m)$, one can see that (\ref{Qn(t)}) provides the following interesting formula : 
\begin{equation} \label{Qn(t)-bis}
\E_x\big[e^{i\, \langle t, S_n\rangle}\, f(X_n)\big] = \big(Q(t)^nf\big)(x)\ \ \ \ (x\in E).
\end{equation}
%Let us show that the so-called Markov random walks 
%and the additive functionals of a Markov chain 
%satisfy Condition~($\cG$). 

\noindent {\it Example: the Markov random walks.} \\[0.1cm] 
If $(X_n,S_n)_{n\in\N}$ is a Markov chain with state space $E\times \R^d$ and transition probability $P$ satisfying the following property 
$$\forall(x,a)\in E\times\R^d,\ \forall A\in\cE,\ \forall B\in B(\R^d),\ \ \ P\big((x,a),A\times B\big) = 
P\big((x,0),A\times (B-a)\big),$$
the sequence $(X_n,S_n)_{n\in\N}$ is called a Markov random walk (MRW). Of course $(X_n)_{n\in\N}$ is then also a Markov chain, called the driving Markov chain of the MRW. We still assume that $S_0=0$. The previous translation property is equivalent to the following one: for any bounded measurable function $F$ on $E\times\R^d$ and for all $a\in\R^d$, we have $(PF)_{a} = P(F_{a})$, where we set for any function $G : E\times\R^d\r \C$: $\, G_{a}(x,b) := G(x,b+a)$.  
Let us check Condition~($\cG$). Let $t\in\R^d$ (fixed). Given $f : E\r\R$ bounded and measurable, we set: $F(x,b) := f(x)\, e^{i\, \langle t, b\rangle}$ for $x\in E$ and $b\in\R^d$. We have: 
\begin{eqnarray*} 
\big(Q_n(t)f\big)(x) := \E_{(x,0)}\big[e^{i\, \langle t, S_n\rangle}\, f(X_n)\big] 
&=& \E_{(x,0)}\big[(PF)(X_{n-1},S_{n-1})\big] \\
&=& \E_{(x,0)}\big[(PF)_{S_{n-1}}(X_{n-1},0)\big] \\
&=& \E_{(x,0)}\big[(PF_{S_{n-1}})(X_{n-1},0)\big] \\
&=& \E_{(x,0)}\big[e^{i\, \langle t, S_{n-1} \rangle}\, (PF)(X_{n-1},0)\big], 
\end{eqnarray*} 
(for the last equality, use: $F_a(y,b) = f(y)\, e^{i\, \langle t, b+a \rangle} =  e^{i\, \langle t , a \rangle}\, F(y,b)$). Therefore we have: $\big(Q_n(t)f\big)(x) = \E_{(x,0)}\big[e^{i\, \langle t, S_{n-1} \rangle}\, g(X_{n-1})\big] = \big(Q_{n-1}(t)g\big)(x)$ with 
%$g(y):=(QF)(y,0)$. 
$$g(\cdot):=(PF)(\cdot,0) = \E_{(\cdot,0)}[f(X_1)\, e^{i\, \langle t, S_1\rangle}] = (Q_1(t)f)(\cdot).$$
We have proved that: $\big(Q_n(t)f\big)(x) = \big(Q_{n-1}(t)\circ Q_1(t)f\big)(x)$, hence Condition~($\cG$) is fulfilled.  

\noindent As a classical example of Markov random walk, recall that, if $(S_n)_{n\in\N}$ is a $d$-dimensional additive functional of a Markov chain $(X_n)_{n\in\N}$, then $(X_n,S_n)_{n\in\N}$ is a MRW. 
%==============================================
\subsection{Spectral methods} \label{nag_method}
When $(X_n,S_n)_{n\geq0}$ satisfies Condition ($\cG$), the spectral method consists in investigating the regularity conditions of  Hypothesis $\cR(m)$ via Formula~(\ref{Qn(t)-bis}). Let us outline the usual spectral method developed in \cite{bab} and the key idea of the weak spectral method presented in \cite{fl}. Given Banach spaces $\cB$ and $X$, we denote by $\cL(\cB,X)$ the space of the bounded linear operators from $\cB$ to $X$, with the simplifications $\cL(\cB) = \cL(\cB,\cB)$ and $\cB' = \cL(\cB,\C)$. 

\noindent {\it (I) The spectral method via the standard perturbation theorem.} \\[0.1cm]
Let $(\cB,\|\cdot\|_{\scriptsize \cB})$ be a Banach space of $\C$-valued measurable functions on $E$. The operator norm in $\cL(\cB)$ is also denoted by $\|\cdot\|_{\scriptsize \cB}$. Let us assume that $(X_n)_{n\geq0}$ possesses a stationary distribution $\pi$ defining a continuous linear form on $\cB$. In \cite{bab}, the assumptions are the following ones: for each $t\in\R^d$, the Fourier operator $Q(t)$ defined in (\ref{noyau-fourier}) is in $\cL(\cB)$, and \\[0.18cm]
\indent (A) $\, Q$ is strongly ergodic on $\cB$, namely~: $\ \lim_n Q^n = \pi$ in $\cL(\cB)$, \\[0.12cm]
\indent (B) $\, $ For any compact set $K\subset \R^d\setminus\{0\}$: $\exists\rho<1,\ \sup_{t\in K}\|Q(t)^n\|_{\scriptsize \cB} = O(\rho^n)$, \\[0.12cm]
\indent (C) $\, Q(\cdot)\in\cC_b^m\big(\R^d,\cL(\cB)\big)$ for some $m\in\N$. \\[0.18cm]
Let $\rho_0\in(\rho,1)$. From (B), $(z-Q(t))^{-1}$ is defined for all $t\in K$ and $z\in\C$ such that $|z|=\rho_0$. By (C), we have: 
$(z-Q(\cdot))^{-1}\in\cC_b^m(K,\cL(\cB))$, and from $Q(t)^n = \frac{1}{2i\pi} \oint_{|z|=\rho_0} z^n(z-Q(t))^{-1}dz$, Hypothesis~$\cR(m)$-(ii) can be then deduced from (\ref{Qn(t)-bis}) when $f\in\cB$. Next, using (A) (C) and the standard perturbation theorem, one can prove that, for $t\in B:=B(0,R)$ for some $R>0$, $Q(t)^n  = \lambda(t)^n\Pi(t) + N(t)^n$ where $\lambda(t)$ is the perturbed eigenvalue near $\lambda=1$, $\Pi(t)$ is the corresponding rank-one eigenprojection, and $N(t)\in\cL(\cB)$ satisfies 
$\sup_{t\in B}\|N(t)^n\|_{\scriptsize \cB} = O(\kappa^n)$ for some $\kappa\in(0,1)$. Since the previous eigen-elements inherit the regularity of $Q(\cdot)$, Hypothesis~$\cR(m)$-(i) is fulfilled. \\[0.12cm]
\noindent  As already mentioned in Section~\ref{intro}, Condition~(C) requires strong assumptions on $Q$ and $S_1$. In fact, by deriving formally (\ref{noyau-fourier}) in the variable $t$ (in case $d=1$ and at $t=0$ to simplify), one gets : $(Q^{(m)}(0)f)(x) := i^m\, \E_x[S_1^m\, f(X_1)]$. Therefore a necessary condition for (C) to be true is that $x\mapsto \E_x[S_1^m\, f(X_1)]\in\cB$ for all $f\in\cB$. 
%Even if $\cB$ is not specified here, one can see that 
This condition may be clearly quite restrictive or even not satisfied when $S_1$ is unbounded.  

\noindent {\it (II) A weak spectral method via the Keller-Liverani perturbation theorem.} \\[0.1cm]
\noindent The following alternative assumptions are proposed in \cite{fl}: (A) (B) hold true with respect to a  whole family $\{{\cB }_\theta,\ \theta\in I\}$ of Banach spaces (instead of a single one as in (I)), and (C) is replaced by the following condition: \\[0.15cm]
\indent (C') $\,$ On each space $\cB _\theta$,  $Q(\cdot)$ satisfies the  Keller-Liverani perturbation theorem \cite{keli} near $t=0$, and $Q(\cdot)\in\cC_b^k\big(\R^d,\cL(\cB_\theta,\cB_{\theta'})\big)$ for suitable $\cB_\theta\subset \cB_{\theta'}\, $ ($k=0,\ldots,m$). \\[0.15cm]
The above regularity hypothesis will involve (in substance) that we have, for some suitable $b, c\in I$, the following condition: $\forall f\in\cB_{b}$, $\, x\mapsto \E_x[S_1^m\, f(X_1)]\in\cB_{c}$ (with $\cB_{b}\subset \cB_{c}$). Thanks to the "gap" between $\cB_{b}$ and $\cB_{c}$, this condition is less restricting than in (I). Of course, the passage from (C') to the regularity properties of the maps $(z-Q(\cdot))^{-1}$ is here more difficult than in the method (I) because $(z-Q(t))^{-1}$ must be seen as elements of $\cL(\cB_\theta,\cB_{\theta'})$ according to a procedure involving several suitable values $(\theta,\theta')\in I^2$. Such results have been presented in \cite{gouliv,fl,seb-08}. 
%=============================================
%=============================================
\section{Application to three classical Markov models} \label{applic}
In this section, the weak spectral method is applied to the three classical models cited in Section~\ref{intro}: the $v$-geometrically ergodic Markov chains, the $\rho$-mixing Markov chains, and the iterative Lipschitz models. In order to use directly some technical results of \cite{fl}, we consider Markov random walks of the form $S_n = \sum_{k=1}^{n} \xi(X_k)$, where $\xi : E\r\R^d$ is a centered functional. 
We establish below that, in the three above models, the property~(\ref{ren-intro}) holds under the nonlattice condition and under some moment conditions on $\xi$ of order $m_d+\varepsilon$, where $m_d$ has been defined in Section~\ref{intro} by $m_d = \max\{d-2,2\}$. First, let us give some common definitions and preliminary properties that will be used in the three models. 

\noindent Let $(X_n)_{n\geq 0}$ be a Markov chain with state space $(E,\cE)$, transition probability $Q(x,dy)$, stationary distribution $\pi$, and initial distribution $\mu$. Let $\xi = (\xi_1,\ldots,\xi_d)$ be a $\R^d$-valued measurable functional  on $E$, and let us define the following associated random walk: \\[0.1cm]
\indent $\displaystyle \ \ \ \ \ \ \ \ \ \ \ \ \ \ \ \ \ \ \ \ \ \ \ \ \ \ \ \ \ \ \ \ \ \ \ \ \ \ \ \ 
S_n = \sum_{k=1}^{n} \xi(X_k)$. \\[0.12cm]
{\it Throughout this section, we assume that $d\geq3$, and that $\xi$ is $\pi$-centered (i.e.~$\xi$ is $\pi$-integrable and $\pi(\xi_i) = 0$ for $i=1,\ldots,d$). }

\noindent Obviously $(X_n,S_n)_{n\in\N}$ is a special case of Markov random walk (with $(X_n)_{n\geq 0}$ as driving chain), and the Fourier operators $Q(t)$ in (\ref{noyau-fourier}) are here given by the kernels 
\begin{equation} \label{fourier-kernels}
Q(t)(x,dy) = e^{i\langle t,\, \xi(y)\rangle}Q(x,dy). 
\end{equation}
All the Banach spaces (say $(\cB,\|\cdot\|_{\scriptsize \cB})$) used in this section are either composed of complex-valued $\pi$-integrable functions on $E$, or composed of classes (modulo $\pi$) of such functions. In addition $\cB$ contains $1_E$ (or $Cl(1_E)$), and $\pi$ always defines  a continuous linear form on $\cB$ (i.e.~$\pi\in\cB'$), so that the following rank-one projection can be defined on $\cB$:  
\begin{equation} \label{Pi}
\Pi f = \pi(f)1_E\ \  (f\in\cB). 
\end{equation}
We will say that $(X_n)_{n\geq 0}$ is strongly ergodic on $\cB$ if $Q$ continuously acts on $\cB$ (i.e.~$Q\in\cL(\cB)$) and satisfies the following property: 
\begin{equation} \label{K1}
\exists \kappa_0<1,\ \exists C>0,\ \forall n\geq1,\ \ \|Q^n-\Pi\|_{\scriptsize \cB} \leq C\, \kappa_0^{n},
\end{equation}
where $\|\cdot\|_{\scriptsize \cB}$ denotes the operator norm on $\cB$. Condition~(\ref{K1}) is equivalent to that already mentioned in Subsection~\ref{nag_method}~: $\lim_n\|Q^n-\Pi\|_{\scriptsize \cB} = 0$. We will repeatedly use the following consequence of Formula~(\ref{Qn(t)-bis}):  
\begin{equation} \label{CF} 
f\in\cB,\ \mu\in\cB'\ \Rightarrow\ \forall n\geq 1,\ \forall t\in\R^d,\ \ \ \E_\mu[f(X_n)\, e^{i\langle t,S_n\rangle}] = \mu(Q(t)^nf).
\end{equation} 
All the next conditions on the initial distribution $\mu$ are satisfied with $\mu = \pi$, and in Subsections~\ref{sub_sec_v_geo} and \ref{sub_sec_ite} when $\mu$ is the Dirac distribution at any $x\in E$. 

\noindent It will be seen that the assumptions of the next corollaries always imply that the following asymptotic covariance matrix is well-defined: 
\begin{equation} \label{sigma}
\Sigma = \lim_n\frac{1}{n}\, \E_\pi[S_nS_n^*]. 
\end{equation}

\noindent Now let us introduce the  following classical (Markov) nonlattice condition:  

\noindent {\it Nonlattice Condition~: $(Q,\xi)$, or merely $\xi$, is said nonlattice if  
there exist no $b\in\R^d$, no closed subgroup $H$ in $\R^d$, $H\neq \R^d$, no $\pi$-full $Q$-absorbing set
$A\in\cE$, and finally no bounded measurable function $\theta\, :\, E\r\R^d$ such that }
$$\forall x\in A, \ \ \xi(y)+\theta(y)-\theta(x)\in b+H\ \ Q(x,dy)-a.s.$$ 
As usual, $A\in\cE$ is said to be $\pi$-full if $\pi(A)=1$, and $Q$-absorbing if $Q(a,A)=1$ for all $a\in A$. 
%\begin{rem} \label{arithmetic} 
%\end{rem}
Finally each partial derivative $\frac{\partial^k Q}{\partial t_{p_1}\cdots\partial t_{p_k}}(t)$ of the Fourier kernels $Q(t)$ (see (\ref{fourier-kernels})) is defined by means of the kernel  
\begin{equation} \label{deri-partial-kernel}
Q_{(p_1,\ldots,p_k)}(t)(x,dy) = i^k\left(\prod_{s = 1}^k \xi_{p_s}(y)\right)e^{i\langle t,\xi(y)\rangle}\, Q(x,dy).
\end{equation}
%===============
%===============
\subsection{Applications to the $v$-geometrically ergodic Markov chains.} \label{sub_sec_v_geo}
We suppose that $\cE$ is countably generated. Let $v : E\r [1,+\infty)$ be some fixed unbounded function, and let us assume that $(X_n)_{n\geq0}$ is $v$-geometrically ergodic, namely~:  we have (\ref{K1}) on the  weighted supremum-normed space $(\cB_v,\|\cdot\|_v)$ of all measurable functions $f : E\r\C$ satisfying  the condition~: 
$$\|f\|_v  = \sup_{x\in E} \frac{|f(x)|}{v(x)} < \infty.$$  
Notice that the previous assumption involves $\pi(v) < \infty$. The $v$-geometrical ergodicity condition can be investigated with the help of the so-called drift conditions. 
For this fact, and for the classical examples of such models, we refer to \cite{mey}.  

\noindent The method (II) outlined in Subsection~\ref{nag_method} is applied here with the following spaces~: for $0<\theta\leq1$, we denote by $(\cB_\theta,\|\cdot\|_\theta)$ the weighted supremum-normed space associated to $v^\theta$, where we set 
$$\|f\|_\theta =  \sup_{x\in E}\frac{|f(x)|}{v(x)^{\theta}}.$$ 
\begin{cor} \label{geo-ren-gene}
Assume that $\mu(v)<\infty$ and that $\xi : E\r\R^d\, $ ($d\geq 3$) is a $\pi$-centered nonlattice functional such that we have with $m_d = \max\{d-2,2\}$: 
\begin{equation} \label{xi-geo}
\exists C>0,\ \exists\, \delta_0>0,\ \ \|\xi\|^{m_d+\delta_0}\leq C\, v.
\end{equation} 
Let $f\in\cB_\vartheta$, $f\geq 0$, where $\vartheta$ is such that $0 < \vartheta < 1-\frac{m}{m_d+\delta_0}$ with $m\in(2,2+\delta_0)$ if $d=3, 4$, and $\, m=d-2\, $ if $\, d\geq 5$. 
\\[0.1cm]
Then we have (\ref{ren-rate-theo}) with $L(0) = \pi(f)$ and $\Sigma$ defined by (\ref{sigma}).  
\end{cor}
\noindent In order to present a general study of Hypothesis~$\cR(m)$ in the present context, we consider below any real numbers $m, \eta$ such that $0<m<\eta$,  we set $\tau = m- \lfloor m\rfloor$, where $\lfloor m\rfloor$ denotes the integer part of $m$, and we assume that $\xi : E\r\R^d$ is a measurable function such that: 
\begin{equation} \label{xi-geo-proof}
\exists C>0,\ \ \|\xi\|^{\eta}\leq C\, v.
\end{equation} 
\noindent Thanks to Theorem~\ref{ren-theo}, Corollary~\ref{geo-ren-gene} follows from the next Propositions~\ref{prop-Rm(i)}-\ref{prop-Rm(ii)} (applied with $\eta = m_d+\delta_0$) and from Remark~\ref{rem-deri-hess}.  

\begin{pro} \label{prop-Rm(i)}
Let $\vartheta_m$ be such that $0<\vartheta_m < \vartheta_m+\frac{m}{\eta} < 1$. Assume that $\xi$ satisfies Condition~(\ref{xi-geo-proof}), that $\mu(v)<\infty$, and that $f\in\cB_{\vartheta_m}$. Then Hypothesis~$\cR(m)$-(i) holds true with $L(0) = \pi(f)$. 
\end{pro}
\begin{proof} 
For any $\kappa\in(0,1)$, let us define 
\begin{equation} \label{cD-kappa}
\cD_\kappa = \big\{z : z\in\C,|z|\geq \kappa,\ |z-1|\geq (1-\kappa)/2\big\}. 
\end{equation}
%The next proposition concerns the spectral properties of the operator $Q(t)$ for $t$ near 0 when $Q(t)$ acting on the Banach spaces $\cB_\theta$, $\theta\in I$. 
\begin{lem} \label{cK-geo} 
Let $\theta\in (0,1]$. Then Property~(\ref{K1}) is satisfied on $\cB_\theta$. Moreover, there exists $\kappa_\theta\in(0,1)$ such that, for all $\kappa\in(\kappa_\theta,1)$, there exists $R_\kappa>0$ such that we have: \\[0.15cm]
\indent $\forall z\in\cD_\kappa,\ \forall t\in B(0,R_\kappa),\ \ (z-Q(t))^{-1}\in\cL(\cB_\theta)$, \\[0.15cm]
\indent $\sup\big\{\, \|(z-Q(t))^{-1}\|_\theta\, :\ z\in\cD_\kappa,\  t\in B(0,R_\kappa)\, \big\}< \infty$. 
\end{lem}
\begin{proof} 
By hypothesis ($v$-geometrical ergodicity), we have (\ref{K1}) on $\cB_1$. The fact that  (\ref{K1}) extends to $\cB_\theta$ is a well-known consequence of the so-called drift conditions and Jensen's inequality, see e.g.~\cite[Lem.~10.1]{fl}. The second assertion follows from the Keller-liverani perturbation theorem \cite{keli}, see \cite[Lem.~10.1 and $\S$~4]{fl}. 
\end{proof}
Now let $\kappa\in(\kappa_{\vartheta_m},1)$. Let $\Gamma_{0,\kappa}$ be the oriented circle centered at $z=0$, with radius $\kappa$, and let $\Gamma_{1,\kappa}$ be the oriented circle centered at $z=1$, with radius $\frac{1-\kappa}{2}$. From (\ref{K1}) and the Keller-liverani theorem (both applied on $\cB_{\vartheta_m}$), we can write for all $t\in B(0,R_\kappa)$ the following equality in $\cL(\cB_{\vartheta_m})$ (e.g.~see \cite[Sec.~7]{fl}): 
\begin{equation} \label{Q_l_L_N}
Q(t)^n  = \lambda(t)^n\Pi(t) + N(t)^n
\end{equation}
where $\lambda(t)$ is the perturbed eigenvalue near $\lambda(0)=1$, and $\Pi(t), N(t)^n$ can be defined by the following line integrals: 
\begin{equation} \label{line-integral}
\Pi(t) = \frac{1}{2i\pi} \oint_{\Gamma_{1,\kappa}} \big(z-Q(t)\big)^{-1}\, dz\ \ \ \mbox{and} \ \ \ N(t)^n = \frac{1}{2i\pi} \oint_{\Gamma_{0,\kappa}} z^n\, \big(z-Q(t)\big)^{-1}\, dz. 
\end{equation}
For any probability measure $\nu$ on $E$ such that $\nu(v) < \infty$ and for any $g\in\cB_{\vartheta_m}$, we define  
\begin{equation} \label{proof-L}
L_{\nu,g}(t) := \nu(\Pi(t)g) = \frac{1}{2i\pi} \oint_{\Gamma_{1,\kappa}} \nu\big((z-Q(t))^{-1}g\big)\, dz,
\end{equation}
\begin{equation} \label{proof-Rn}
\ \ \ R_{n,\nu,g}(t) := \nu(N(t)^ng) = \frac{1}{2i\pi} \oint_{\Gamma_{0,\kappa}} z^n\, \nu\big((z-Q(t))^{-1}g\big)\, dz.
\end{equation}
The previous quantities are well defined since $g\in\cB_{\vartheta_m}$ and $\nu\in \cB_1'\subset\cB_{\vartheta_m}'$. 
%(the reason why we take $\mu\in\cB_1'$ will be explained below). 
From (\ref{CF}) we have (\ref{X_n-S-n}) with $\forall t\in B(0,R_\kappa)$: 
\begin{equation} \label{L_Rn}
L(t) = L_{\mu,f}(t)\ \ \ \mbox{and}\ \ \ R_n(t) = R_{n,\mu,f}(t). 
\end{equation}
%from $Q(t) = \lambda(t)\Pi(t) + N(t)$ and $\lim_{t\r0}\pi\big(\Pi(t)1_E\big) =  \pi(\Pi1_E) = 1$ 
In particular we have $L(0) = \mu(\Pi f)= \pi(f)$. Besides it can be easily proved (see \cite[Sect.7]{fl} for details) that we have with possibly $R_\kappa$ reduced:
\begin{equation} \label{lamda}
\forall t\in B(0,R_\kappa),\ \lambda(t) = \frac{\pi(Q(t)1_E) - R_{1,\pi,1_E}(t)}{L_{\pi,1_E}(t)}. 
\end{equation}
Finally, from Lemma~\ref{cK-geo} and (\ref{line-integral}), there exists $C>0$ such that: $\forall t\in B(0,R_\kappa),\ R_n(t) \leq C\, \kappa^n$. 
Therefore, the series $\sum_{n\geq1}R_n(t)$ converge, and we have 
\begin{equation} \label{def-cR(t)}
\sum_{n\geq1} R_n(t) = \frac{1}{2i\pi}\, \oint_{\Gamma_{0,\kappa}} \frac{z}{1-z}\, \mu\big((z-Q(t))^{-1}f\big)\, dz.
\end{equation}
From (\ref{L_Rn}) (\ref{lamda}) (\ref{def-cR(t)}), the desired regularity properties in Hypothesis~$\cR(m)$-(i) will be fulfilled if we prove that $\pi(Q(\cdot)1_E)$, $\, \pi((z-Q(\cdot))^{-1}1_E)$, and $\mu((z-Q(\cdot))^{-1}f)$ are in $\cC_b^{m}(B(0,R),\C)$ for some $R>0$ and have uniform bounded derivatives in $z\in\cD_\kappa$ (thus in $z\in\Gamma_{0,\kappa}\cup\Gamma_{1,\kappa}$).  
Let $Q_k$ denote any kernel of the form $Q_{(p_1,\ldots,p_k)}$ defined in (\ref{deri-partial-kernel}). 
\begin{lem} \label{deri-geo} 
Let $0<\theta<\theta'\leq1$, let $k\in\N$ such that $0\leq k\leq \lfloor m\rfloor$, let $R>0$, and set $B=B(0,R)$. Then we have the following properties~:  \\[0.15cm]
%(i) If $\theta + \frac{k}{\eta} < \theta'$, then $Q_k\in\cC_b^0(B,\cL(\cB_\theta,\cB_{\theta'}))$, \\[0.15cm]
(i) If $\tau'\in(0,1)$ and $\theta + \frac{k+\tau'}{\eta} \leq \theta'$, then $Q_k\in\cC_b^{\tau'}(B,\cL(\cB_\theta,\cB_{\theta'}))$,  \\[0.15cm]
(ii) If $k \leq \lfloor m\rfloor-1$ and $\theta + \frac{k+1}{\eta}  < \theta'$, then 
$Q_k\in\cC_b^1(B,\cL(\cB_\theta,\cB_{\theta'}))$. 
\end{lem}
\begin{proof}
Lemma~\ref{deri-geo} follows from (\ref{xi-geo-proof}): this is established in \cite[Lem.~10.4-5]{fl}. 
\end{proof}
To make easier the use of Lemma~\ref{deri-geo}, let us define for any fixed $u>0$: $T_u(\theta) := \theta + u/\eta$. 
Since $m = \lfloor m\rfloor +\tau$ and $\vartheta_m + m/\eta < 1$, one can choose $\tau'\in(\tau,1)$ and $\delta>0$ such that 
$$T_{\tau'}T_{1+\delta}^{\lfloor m\rfloor}(\vartheta_m) := \vartheta_m + \frac{\lfloor m\rfloor}{\eta} + \frac{\tau'}{\eta} + \frac{\lfloor m\rfloor\delta}{\eta} = 1.$$
Lemma~\ref{deri-geo} shows that $Q(\cdot)\in\cC_b^m(B,\cL(\cB_{\vartheta_m},\cB_1))$. Since $1_E\in\cB_{\vartheta_m}$ and $\pi\in \cB_1'$, one gets:  $\pi(Q(\cdot)1_E)\in\cC_b^m(B,\C)$. Lemma~\ref{derivation} below provides the above claimed regularity properties involving the resolvent $(z-Q(t))^{-1}$, and thus completes the proof of Proposition~\ref{prop-Rm(i)}. 
\end{proof}
\begin{lem} \label{derivation}
There exists $\kappa\in(\kappa_{\vartheta_m},1)$ and $R>0$ such that, for any probability measure $\nu$ on $E$ satisfying $\nu(v) < \infty$ and for any $g\in\cB_{\vartheta_m}$, the map $t\mapsto r_z(t) := \nu((z-Q(t))^{-1}g)$ is in $\cC_b^{m}(B(0,R),\C)$ for all $z\in\cD_\kappa$, and we have 
$$\forall \ell = 0,\ldots,\lfloor m\rfloor, \ \
%\cR_\ell := 
\sup\big\{|r_z^{(\ell)}(t)|,\, z\in\cD_\kappa,\, t\in B(0,R)\,  \big\}<\infty,$$
where $r_z^{(\ell)}(\cdot)$ denotes any partial derivative of order $\ell$ of $r_z(\cdot)$.  Moreover,  if $m\notin\N$, then the $\tau$-Hölder coefficient of $r_z^{(\lfloor m\rfloor)}$on $B(0,R)$ is uniformly bounded in $z\in\cD_\kappa$. 
\end{lem}
\begin{proof} 
On the basis of Lemmas~\ref{cK-geo}-\ref{deri-geo}, Lemma~\ref{derivation} is established in \cite[Prop.~10.3]{fl} in the case $m\in\N$. The case $m\notin\N$ can be obtained by slightly extending the method of \cite{fl}. It can be also deduced from \cite[Th.~3.3]{seb-08} which specifies and generalizes the Taylor expansions obtained in \cite{aap,gouliv}. Let us verify that the assumptions of \cite[Th.~3.3]{seb-08} are fulfilled. Define the spaces $\cB^{\lfloor m\rfloor+1}:= \cB_{\vartheta_m}$, $\, \cB^0 = \cB_{T_{\tau'}T_{1+\delta}^{\lfloor m\rfloor}(\vartheta_m)} = \cB_1$ and 
$$\forall j=1,\ldots,\lfloor m\rfloor,\ \ \ \ \cB^{j} := \cB_{T_{1+\delta}^{\lfloor m\rfloor+1-j}(\vartheta_m)}.$$
Note that $\cB^{\lfloor m\rfloor+1} \subset \cB^{\lfloor m\rfloor} \subset \ldots \subset \cB^1 \subset \cB^0$. Then Lemmas~\ref{cK-geo}-\ref{deri-geo} and \cite[Th.~3.3]{seb-08} show that there exist $\kappa\in(\kappa_{\vartheta_m},1)$ and $R>0$ such that $r_z(t)$ admits a Taylor expansion of order $\lfloor m\rfloor$ at any point $t_0\in B(0,R)$, with a remainder $O(|t-t_0|^m)$ which is uniform in $z\in\cD_\kappa$ and $t_0\in B(0,R)$. As already observed in \cite[Rk.~3.6]{seb-08}, the passage to the $m$-differentiability properties of $r_z(\cdot)$ can be deduced from a general lemma in \cite{camp}. 
\end{proof} 
%========

\begin{pro} \label{prop-Rm(ii)}
Assume that $\xi$ is nonlatice and satisfies Condition~(\ref{xi-geo-proof}), that $\mu(v)<\infty$, and that $f\in\cB_{\vartheta_m}$, where $\vartheta_m$ is defined as in Proposition~\ref{prop-Rm(i)}. Then Hypothesis~$\cR(m)$-(ii) is fulfilled. 
\end{pro}
\begin{proof}
Let $0<r<r'$. The nonlattice assumption yields the following result (see \cite[Lem.~10.1, Prop.~5.4]{fl}): 
\begin{lem} \label{cS-geo} 
For any $\theta\in(0,1]$, there exists $\rho_\theta<1$ 
such that: $\sup_{t\in K_{r,r'}}\|Q(t)^n\|_{\theta} = O(\rho_\theta^n)$, where $\|\cdot\|_{\theta}$ stands here for the operator norm on $\cB_\theta$. 
\end{lem}
Lemma~\ref{cS-geo} and (\ref{CF}) imply that the series $\cE(t) := \sum_{n\geq1}\E_\mu[f(X_n)\, e^{i\langle t, S_n\rangle}]$ uniformly converge on $K_{r,r'}$ (use $f\in\cB_{\vartheta_m}$ and $\mu\in  \cB_1'\subset\cB_{\vartheta_m}'$). Setting $\theta_j:= T_{1+\delta}^{\lfloor m\rfloor-j}(\vartheta_m)$ for $j=0,\ldots,\lfloor m\rfloor$, we define $\rho = \max\{\rho_1,\rho_{\theta_0},\ldots,\rho_{\theta_{\lfloor m\rfloor}}\}$. Let $\rho_0\in(\rho,1)$, and let $\Gamma$ denote the oriented circle centered at $z=0$, with radius $\rho_0$. Since $Q(t)^n = \frac{1}{2i\pi} \oint_{\Gamma} z^n\big(z-Q(t)\big)^{-1}dz$ in $\cL(\cB_{\vartheta_m})$ for all $t\in K_{r,r'}$, we have  
$$\cE(t) = \sum_{n\geq1} \mu\big(Q(t)^nf\big) = \frac{1}{2i\pi} \oint_{\Gamma} \frac{z}{1-z}\, \mu\big((z-Q(t))^{-1}f\big)\, dz.$$
The desired regularity of $\cE(\cdot)$ 
%in Hypothesis~$\cR(m)$-(ii) again 
can be then deduced as in the proof of Proposition~\ref{prop-Rm(i)}. 
%from \cite[Th.~3.3]{seb-08} (applied with the same spaces as above), completed by \cite{camp} in order to pass from the Taylor expansion to the claimed differentiability property.  
\end{proof}
%===============
\begin{rem} \label{rem-deri-hess} 
If $\xi$ is $\pi$-centered and satisfies (\ref{xi-geo-proof}) with $\eta > 2$, then we have 
$\nabla\lambda(0) = 0$. If $\mu(v)<\infty$, then we have: $-\, Hess\, \lambda(0) = \lim_n\frac{1}{n}\, \E_\mu[S_nS_n^*]$. Finally, if $\xi$ is  nonlattice, then $\Sigma := - Hess\, \lambda(0)$ is positive definite. 
\end{rem}
\noindent{\it Proof of Remark~\ref{rem-deri-hess}.} Assume that $\mu(v)<\infty$. We have $\E_\mu[\|S_n\|^2] < \infty$  (use (\ref{xi-geo-proof})). From (\ref{proof-Rn}) and Lemma~\ref{derivation}, it follows that $\sup_{n\geq1}|R_{n,\mu,1_E}^{(\ell)}(0)| < \infty$ for $\ell=1,2$.  
By Proposition~\ref{deri-hess}(i), we get: $\nabla\lambda(0)/i = \lim_n\frac{1}{n}\, \E_\mu[S_n]$. When applied with $\mu=\pi$, this gives $\nabla\lambda(0) = 0$ since $\xi$ is $\pi$-centered. The second point of Remark~\ref{rem-deri-hess} follows from Proposition~\ref{deri-hess}(ii), and the last assertion is established in \cite[Sect.~5.2]{fl}. 
\fdem
\indent The following remarks also apply to the examples of the two next subsections.  
%\begin{rem} \label{use-K-L}
%The use of the Keller-Liverani perturbation theorem \cite{keli}, needed in the proof of Lemma~\ref{cK-geo} and (\ref{Q_l_L_N}), is essential in the above spectral method. In fact, the standard perturbation theorem yields (\ref{Q_l_L_N}) provided that $\lim_{t\r0}Q(t) = Q$ in the operator norm topology. Unfortunately, this condition is not fulfilled in general when $\xi$ is unbounded, see \cite[Sect.~3]{fl}. The continuity condition in \cite{keli} is much weaker than the previous one; on the other hand, the $Q(t)$'s have to satisfy some Doeblin-Fortet inequalities in \cite{keli}. 
%\end{rem}
%\begin{rem} \label{rem-proc-deir-fl}
%An adaptation of \cite[Sect.~7]{fl} can be also employed to prove directly (i.e.~without using \cite{camp})) Lemma~\ref{derivation} and the regularity property of Hypothesis~$\cR(m)$-(ii).\footnote{To that effect, take $\delta>0$ such that $T_\tau\, T_{1+\delta}^{\lfloor m\rfloor}\, T_\delta^{\lfloor m\rfloor+1}(\vartheta_m) = \vartheta_m + m/\eta + (2\lfloor m\rfloor+1)\delta/\eta = 1$, and use the family $\{{\cB }_\theta,\ \theta\in [\vartheta_m,1]\}$.} 
%\end{rem}
%\begin{rem} \label{dep-mu-f}
%The functions $L(\cdot)$ and $R_n(\cdot)$ defined in (\ref{L_Rn}) both depend on $\mu$ and $f$, while $\lambda(\cdot)$ in (\ref{lamda}) is independent from $\mu, f$. 
%\end{rem}
\begin{rem} \label{delta0}
Even in case $d\geq5$ for which Hypothesis $\cR(m_d)$ is needed in Theorem~\ref{ren-theo} with $m_d=d-2$, the above use of \cite[Th.~3.3]{seb-08} (or \cite[Sect.~7]{fl}) does not allow to prove (\ref{ren-rate-theo}) under the moment assumption $\|\xi\|^{m_d}\leq C\, v$. This is due to the fact that, when $\xi$ is unbounded, the property $Q\in\cC^\varepsilon(\R^d,\cL(\cB_\theta))$ is not fulfilled in general and must be replaced in the derivation procedure (of both \cite{fl,seb-08}) by $Q\in\cC^\varepsilon(\R^d,\cL(\cB_\theta,\cB_{\theta'}))$ with $\theta<\theta'$. This yields a "gap" between the spaces $\cB_{\vartheta_m}$ and $\cB_v$ (used in Lemma~\ref{derivation}) which is slightly bigger than the expected one. This is the reason why the order in (\ref{xi-geo}) is $m_d+\delta_0$ with some arbitrary small $\delta_0>0$. 
\end{rem}
\begin{rem} \label{exten-cond-G}
%By adapting the present method, t
Property~(\ref{ren-rate-theo}) could be investigated in the general setting of the sequences $(X_n,S_n)_{n\geq0}$ satisfying Condition ($\cG$) of Subsection~\ref{subsect_semi_group}, provided that the driving Markov chain $(X_n)_{n\geq0}$ is $v$-geometrically ergodic. The moment conditions then focus on $(S_1,X_1)$. 
\end{rem}
%===========================================================================
%\subsection{Applications to the strongly ergodic Markov chains on $\L^2(\pi)$} \label{sub_sec_L2}
\subsection{Applications to the $\rho$-mixing  Markov chains} \label{sub_sec_L2}
Let us assume that the $\sigma$-field $\cE$ is countably generated, and that $(X_n)_{n\geq0}$ is a $\rho$-mixing Markov chain. Equivalently (see \cite{rosen}), $(X_n)_{n\geq0}$ possesses a stationary distribution $\pi$ and satisfies the strong ergodicity condition (\ref{K1}) on the usual Lebesgue space $\L^2(\pi)$. For instance, this condition is fulfilled when $(X_n)_{n\geq0}$ is uniformly ergodic (i.e.~satisfies (\ref{K1}) on the space $\cB^{^\infty}$of bounded measurable complex-valued functions on $E$, or equivalently $(X_n)_{n\geq0}$ is aperiodic, ergodic, and satisfies the so-called Doeblin condition). 
\noindent Let us first present a statement in the stationary case. 
\begin{cor} \label{unif-ren} 
If $\xi : E\r\R^d\, $ ($d\geq 3$) is a $\pi$-centered nonlattice functional such that  
\begin{equation} \label{xi-unif}
\exists\, \delta_0>0,\ \ \pi\big(\|\xi\|^{m_d+\delta_0}\big) < \infty,
\end{equation} 
then we have (\ref{ren-intro}) in the stationary case ($\mu=\pi$) with $\Sigma$ defined by (\ref{sigma}).
\end{cor}
In comparison with the i.i.d.~case, the moment condition (\ref{xi-unif}) is the expected one (up to $\delta_0>0$). Corollary~\ref{unif-ren} is a special case of the next one. 
We denote by $\L^p(\pi)$, $1\leq p\leq\infty$, the usual Lebesgue space associated to $\pi$. 
\begin{cor} \label{unif-ren-gene} 
Assume that $\xi$ satisfies the assumptions of Corollary~\ref{unif-ren} 
and that the initial distribution $\mu$ is of the form 
$\mu = \phi\, d\pi$, where $\phi\in\L^{r'}(\pi)$ for some $r'>\frac{\eta}{\eta-m}$, where $\eta = m_d+\delta_0$,  $m\in(2,2+\delta_0)$ if $d=3,4$, and $m=d-2$ if $d\geq 5$. If $s>\frac{\eta}{\eta-m}$ and $\frac{\eta s}{\eta+ms} > \frac{r'}{r'-1}$, then, for any $f\in\L^s(\pi)$, $f\geq 0$, we have (\ref{ren-rate-theo}) 
with $L(0) = \pi(f)$ and $\Sigma$ defined by (\ref{sigma}).  
\end{cor}
\noindent Given $r'>\eta/(\eta-m)$, the two last conditions on $s$ are  fulfilled if $s$ is sufficiently large. Indeed, let us set $r=r'/(r'-1)$ (ie.~$1/r+1/r' = 1$). Since $1/r'<1-m/\eta$, we have $1<r<\eta/m$, and since 
$\frac{\eta s}{\eta+ms}\nearrow\frac{\eta}{m}$ when $s\r+\infty$, we have $\eta s/(\eta+ms) > r$ when $s$ is large enough. 

\noindent Thanks to Theorem~\ref{ren-theo}, Corollary~\ref{unif-ren-gene}  is a consequence of the next Proposition~\ref{prop-Rm-Lp}. 
Let us consider any real numbers $m, \eta$ such that $0<m<\eta$, set  $\tau = m- \lfloor m\rfloor$, and assume that  $\xi : E\r\R^d$ is a measurable function such that: 
\begin{equation} \label{xi-unif-proof}
\pi\big(\|\xi\|^{\eta}\big) < \infty. 
\end{equation} 
Let $r'>\frac{\eta}{\eta-m}$,  $r : =\frac{r'}{r'-1}$, and let $s$ be such that $s>\frac{\eta}{\eta-m}$ and $\frac{\eta s}{\eta+ms} > r$.  Note that $r<s$. 
\begin{pro} \label{prop-Rm-Lp}
Assume that we have (\ref{xi-unif-proof}), that $\mu = \phi\, d\pi$ with $\phi\in\L^{r'}(\pi)$, and that $f\in\L^s(\pi)$, with $r'$ and $s$ above defined. Then Hypothesis $\cR(m)$-(i) holds true with $L(0) = \pi(f)$. If in addition  $\xi$ is nonlatice, then Hypothesis~$\cR(m)$-(ii) is fulfilled. 
\end{pro}
\begin{proof}
We are going to apply the procedure of the previous subsection to some suitable spaces chosen in the family $\{\cB_\theta := \L^\theta(\pi),\ \theta\in(1,+\infty)\}$. First the conclusions of Lemma~\ref{cK-geo} are true w.r.t.~these spaces, see \cite{rosen} and \cite[Prop.~4.1]{fl}. Second, from the nonlattice condition, Lemma~\ref{cS-geo} is valid , see \cite[Prop.~5.4, Sec.~5.3]{fl}. Now let us define for any fixed $u>0$: $T_u(\theta) := \eta\theta/(\eta+u\theta)$. Then, by using (\ref{xi-unif-proof}) and an easy extension of \cite[Lem.~7.4]{fl}, we can prove that we have for any $\theta\in I := [r,s]$ and $\tau'\in(0,1)$: \\[0.15cm]
\indent $\forall j=1,...,\lfloor m\rfloor$:  $\ T_1^j(\theta)\in I$\ \ $\Rightarrow$\ \ $Q(\cdot)\in \cC_b^j\big(B,\cL(\L^{\theta},\L^{T_1^j(\theta)})\big)$ \\[0.12cm]
\indent $\forall j=0,...,\lfloor m\rfloor$:  $\ T_{\tau'} T_1^j(\theta)\in I$\ \ $\Rightarrow$\ \ $Q(\cdot)\in \cC_b^{j+\tau'}\big(B,\cL(\L^{\theta},\L^{T_{\tau'} T_1^j(\theta)})\big)$, 

\noindent where $B=B(0,R)$ (for any $R>0$). Next, from $r<\frac{\eta s}{\eta+ms}$, one can fix $\tau'\in(\tau,1)$ such that 
$$r \leq  T_{\tau'}\, T_1^{\lfloor m\rfloor}(s) = \frac{\eta s}{\eta + (\lfloor m\rfloor +\tau')s}.$$
By using the spaces $\cB^{\lfloor m\rfloor+1}:= \L^s(\pi)$, $\, \cB^0 = \L^r(\pi)$ and 
$$\forall j=1,\ldots,\lfloor m\rfloor,\ \ \ \ \cB^{j} := \L^{T_{1}^{\lfloor m\rfloor+1-j}(s)}(\pi),$$
one can prove as in Subsection~\ref{sub_sec_v_geo} that the conclusions of Lemma~\ref{derivation} are fulfilled for any probability measure $\nu$ on $E$ defining a continuous linear form on $\L^r(\pi)$ and for any $g\in\L^s(\pi)$. Then the two conclusions of Proposition~\ref{prop-Rm-Lp} can be proved as in Subsection~\ref{sub_sec_v_geo}.  
\end{proof}

\noindent To complete the proof of Corollary~\ref{unif-ren-gene}, notice that, if (\ref{xi-unif-proof}) holds with $\eta>2$, then the conclusions of Remark~\ref{rem-deri-hess} are fulfilled (here the condition on $\mu$ is that of Proposition~\ref{prop-Rm-Lp}). 
%====================================================================
\subsection{Applications to iterative  Lipschitz models} \label{sub_sec_ite}
Here $(E,d)$ is a non-compact metric space in which every 
closed ball is compact, and it is endowed with its Borel $\sigma$-field $\cE$. Let $(G,\cG)$ be a measurable space, 
let $(\varepsilon_n)_{n\geq 1}$ be a  i.i.d.~sequence of random variables taking values in 
$G$, and let $F : E\times G\r E$ be a measurable function. 
Given an initial $E$-valued r.v.~$X_0$ independent of $(\varepsilon_n)_{n\geq1}$, the random iterative  model associated to 
$(\varepsilon_n)_{n\geq 1}$, $F$ and $X_0$ is defined by  (see \cite{duf}) 
$$X_n = F(X_{n-1},\varepsilon_n),\ \ n\geq 1.$$
Let us consider the two following random variables which are classical in these models \cite{duf}~: 
$$\cC : = \sup\bigg\{\frac{d\big(F(x,\varepsilon_1),F(y,\varepsilon_1)\big)}{d(x,y)},\ x,y\in E,\ x\neq y\bigg\}\ \ \ \mbox{and}\ \ \ \ 
\cM  = 1 + \cC + d\big(F(x_0,\varepsilon_1),x_0\big)$$
where $x_0$ is some fixed point in $E$. It is well-known that, if $\cC<1$ almost surely and if 
\begin{equation} \label{moment-ite}
\exists s\geq0,\ \exists \delta_0>0,\ \ \ \E[\, \cM^{(s+1)m_d+\delta_0}\, ] < \infty, 
\end{equation} 
then $(X_n)_{n\geq 0}$ possesses a stationary distribution $\pi$ such that $\pi\big(d(\cdot,x_0)^{(s+1)m_d + \delta_0}\big) < \infty$. 
\begin{cor} \label{ite-ren} 
Let us assume that $\cC<1$ almost surely, that (\ref{moment-ite}) holds, that the initial distribution $\mu$ is such that $\mu\big(d(\cdot,x_0)^{(s+1)m_d + \delta_0}\big) < \infty$, and finally that $\xi : E\r\R^d$ ($d\geq3$) is a $\pi$-centered nonlattice function satisfying the following condition: 
\begin{equation} \label{xi-iterative} 
\exists S\geq 0,\ \forall (x,y)\in E\times E,\ \ \|\xi(x)-\xi(y)\| \leq S\, d(x,y)\, \big[1+d(x,x_0)+d(y,x_0)\big]^s,
\end{equation} 
where $s$ is the real number in (\ref{moment-ite}). Then we have (\ref{ren-intro}) with $\Sigma$ defined by (\ref{sigma}). 
\end{cor}
Notice that (\ref{xi-iterative}) corresponds to the general weighted-Lipschitz condition introduced in \cite{duf}. Some weaker assumptions on $\cC$ are presented at the end of this subsection, as well as some (a priori) less restrictive condition than the nonlattice assumption.   

\noindent {\it Example: the linear autoregressive models in $\R^3$.}  As an illustration, let us consider in $E=\R^3$, equipped with some norm $\|\cdot\|$, the autoregressive model $X_n = A X_{n-1} +\varepsilon_n$,  
where $X_0,\varepsilon_1,\varepsilon_2,\ldots$ are independent $\R^3$-valued random variables, and $A$ is a contractive matrix of order 3 ($\sup_{\|x\|\leq1} \|Ax\|<1$). Then, taking the distance $d(x,y) = \|x-y\|$, $(X_n)_n$ is an iterative Lipschitz model and we have $\cC<1$. 
Let us consider the centered random walk $S_n = X_1+\ldots+X_n-n\E_\pi[X_0]$ 
associated to the functional $\xi(x) := x-\E_\pi[X_0]$. Note that $\xi$ satisfies (\ref{xi-iterative}) with $s=0$. Assume that $\xi$ is nonlattice and that we have 
\begin{equation} \label{moment-auto}
\exists \delta_0>0,\ \ \ \E\big[\, \|X_0\|^{2 + \delta_0} + \|\varepsilon_1\|^{2+\delta_0}\big] <\infty.
\end{equation} 
Then, from Corollary~\ref{ite-ren}, the above random walk $S_n$ satisfies (\ref{ren-intro}). 
This statement extends the 3-dimensional renewal theorem of the i.i.d.~case (obtained here in the special case $A=0$) under the same moment condition up to $\delta_0>0$. Also notice that, except the moment condition on $\varepsilon_1$ in (\ref{moment-auto}), the previous result does not require any special assumption on the law of $\varepsilon_1$. Anyway, it remains true if $X_n = A_n X_{n-1} +\varepsilon_n$ where $(A_n)_{n\geq1}$ is a sequence of r.v.~taking values in the set of matrices of order 3, such that $\|A_1\|<1$ a.s. and $(A_n,\varepsilon_n)_{n\geq1}$ is i.i.d.~and independent of $X_0$. 

\noindent Corollary~\ref{ite-ren} follows from the more general Corollary~\ref{ite-ren-gene} below, that will be proved by applying the operator-type method to the weighted H\"older-type spaces $\cB_{\alpha,\beta,\gamma}$ defined as follows. For $x\in E$, set $p(x) = 1+ \, d(x,x_0)$, 
and given any $0<\alpha\leq 1$ and $0<\beta\leq\gamma$, define for $(x,y)\in E^2$~: 
$$\Delta_{\alpha,\beta,\gamma}(x,y) = p(x)^{\alpha\gamma}\, p(y)^{\alpha\beta} + p(x)^{\alpha\beta}\,   p(y)^{\alpha\gamma}.$$
The space $\cB_{\alpha,\beta,\gamma}$ is by definition composed of the $\C$-valued functions $f$ on $E$ such that
$$m_{\alpha,\beta,\gamma}(f) =  
\sup\bigg\{\frac{|f(x)-f(y)|}{d(x,y)^\alpha\, \Delta_{\alpha,\beta,\gamma}(x,y)},\ x,y\in E,\  
x\neq y\bigg\}\, <\, \infty.$$
Set  $\displaystyle \ |f|_{\alpha,\gamma}  = \sup_{x\in E}\  
\frac{|f(x)|}{p(x)^{\alpha(\gamma+1)}}$ and
$\|f\|_{\alpha,\beta,\gamma} = m_{\alpha,\beta,\gamma}(f) + |f|_{\alpha,\gamma}$. Then $(\cB_{\alpha,\beta,\gamma},\|\cdot\|_{\alpha,\beta,\gamma})$  is a Banach space. 
\begin{cor} \label{ite-ren-gene} Let us assume that the assumptions of Corollary~\ref{ite-ren} hold, and let us fix any real number $\alpha$ such that $0<\alpha\leq\min\{1,\frac{\delta_0}{4(s+1)}\}$. Then we have (\ref{ren-rate-theo}), with $L(0) = \pi(f)$ and $\Sigma$ defined by (\ref{sigma}), for each nonnegative function  $f\in\cB_{\alpha,\vartheta,\vartheta}$, where $\vartheta$ is some real number such that $\vartheta > s+1$ (the condition on $\vartheta$ is specified below in function of $\alpha$ and $\delta_0$). 
\end{cor}
\noindent Corollary~\ref{ite-ren-gene} is a consequence of Theorem~\ref{ren-theo} and of the next Proposition~\ref{prop-Rm-ite} which, given $m\in\N$, gives Hypothesis~$\cR(m+\delta)$ for some $\delta>0$ (this is convenient for our purpose).  

\noindent We assume below that $\cC<1$ a.s.~and that we have: 
\begin{equation} \label{moment-ite-proof}
\exists m\in\N,\ \exists s\geq0,\ \exists \delta_0>0,\ \ \ \E[\, \cM^{(s+1)m+\delta_0}\, ] < \infty, 
\end{equation} 
and that  $\xi : E\r\R^d$ is a measurable function satisfying (\ref{xi-iterative}) with $s$ given in (\ref{moment-ite-proof}). 

\noindent Let $k\in\N$, $\, \alpha\in(0,1]$, $\, \delta'\in(0,\alpha)$, and let $\beta$ be such that $s+1\leq\beta\leq\gamma$. Set $B=B(0,R)$ (for any $R>0$). 
\begin{lem} \label{deri-ite}  
Let us assume that $\gamma' \geq \gamma + \frac{(s+1)(k+\delta')}{\alpha}$ and $\E[\, \cM^{\alpha(\gamma' + \beta)}\, ] < \infty$. Then we have   $Q(\cdot)\in\cC_b^{k+\delta'}\big(B,\cL(\cB_{\alpha,\beta,\gamma},\cB_{\alpha,\beta,\gamma'})\big)$.   
\end{lem} 
Lemma~\ref{deri-ite} can be proved by using the arguments of \cite[Lem.~B.4-4']{fl}. Now, let us define for any $u>0$: $T_u(\gamma) := \gamma + (s+1)u/\alpha$. Let $\alpha$ be fixed such that  $0<\alpha\leq\min\{1,\frac{\delta_0}{4(s+1)}\}$. Set $\vartheta : = s+1+\delta'$ for some $\delta'\in(0,\alpha)$ (specified below), and define: \\[0.14cm]
\indent $\ \ \ \ \ \ \ \ \ \ \ \ \gamma_m := T_{\delta'}T_{1+\delta'}^{m}(\vartheta) = s+1+\delta' +  (s+1)\big(m(1+\delta')+\delta'\big)/\alpha$. \\[0.14cm]
We have $\alpha(\gamma_m+\vartheta) = (s+1)m + 2\alpha(s+1) + 2\alpha\delta' + \delta'(s+1)(m + 1) $. Now, using $\alpha\leq \frac{\delta_0}{4(s+1)}$, we can choose $\delta'\in(0,\alpha)$ sufficiently small such that $\alpha(\gamma_m+\vartheta) \leq (s+1)m + \delta_0$. 
\begin{pro} \label{prop-Rm-ite}
Let us assume that we have $\cC<1$ a.s.~and (\ref{moment-ite-proof}), that $\xi$ satisfies (\ref{xi-iterative}) with $s$ given in (\ref{moment-ite-proof}), that $\mu\big(d(\cdot,x_0)^{(s+1)m + \delta_0}\big) < \infty$, and finally that $f\in\cB_{\alpha,\vartheta,\vartheta}$. Then, for any $\delta\in(0,\delta')$, Hypothesis~$\cR\big(m+\delta\big)$-(i) holds with $L(0) = \pi(f)$. If in addition $\xi$ is nonlatice, then Hypothesis~$\cR\big(m+\delta\big)$-(ii) is fulfilled. 
\end{pro}
\begin{proof}
We apply the procedure of Subsection~\ref{sub_sec_v_geo} with the spaces $\cB_\gamma:=\cB_{\alpha,\vartheta,\gamma}$, $\, \gamma\in [\vartheta,\gamma_m]$, with $\alpha,\vartheta,\gamma_m$ fixed above. First, Lemma~\ref{cK-geo} is valid on each $\cB_\gamma$, see \cite[Prop.~11.2-11.4]{fl}. Second, from the nonlattice condition, Lemma~\ref{cS-geo} applies on each $\cB_\gamma$, see \cite[Prop.~11.8]{fl}. Next, observe that $\gamma_m > \vartheta + (s+1)(m+\delta')/\alpha$. Then, by using Lemma~\ref{deri-ite} and the spaces $\cB^{\lfloor m\rfloor+1}:= \cB_{\alpha,\vartheta,\vartheta}$, $\, \cB^0 = \cB_{\alpha,\vartheta,\gamma_m}$ and 
$$\forall j=1,\ldots,\lfloor m\rfloor,\ \ \ \ \cB^{j} := \cB_{\alpha,\vartheta,T_{1+\delta}^{m+1-j}(\vartheta)},$$
one can prove as in Subsection~\ref{sub_sec_v_geo} that Lemma~\ref{derivation} holds with here any $g\in\cB_{\alpha,\vartheta,\vartheta}$ and any probability measure $\nu$ on $E$ defining a continuous linear form on $\cB_{\alpha,\vartheta,\gamma_m}$. Since  $\alpha(\gamma_m+1) \leq (s+1)m + \delta_0$, the last condition holds if $\mu\big(d(\cdot,x_0)^{(s+1)m + \delta_0}\big) < \infty$. Then the  conclusions of Proposition~\ref{prop-Rm-ite} can be  established as in Subsection~\ref{sub_sec_v_geo}.  
\end{proof}
\noindent To complete the proof of Corollary~\ref{ite-ren-gene}, notice that, under the assumptions of Proposition~\ref{prop-Rm-ite} with $m\geq2$ in (\ref{moment-ite-proof}), the conclusions of Remark~\ref{rem-deri-hess} are fulfilled. 
\begin{rem}
By repeating the previous proof with a more precise use of the results of \cite{fl}, one can establish that under the following assumption (weaker than (\ref{moment-ite}) and $\cC<1$)  
%\begin{equation} \label{mean-cont-moment}
$$\exists s\geq0,\ \exists \delta_0>0,\ \ \E\big[\, (1+\cC^\alpha)\cM^{(s+1)m_d + \delta_0}\, \big] < \infty,\ \ 
\E\big[\, \cC^\alpha\, \max\{\cC,1\}^{(s+1)m_d + \delta_0}\, \big] < 1,$$
%\end{equation}
the conclusion of Corollary~\ref{ite-ren-gene} remains true if $\mu\big(d(\cdot,x_0)^{(s+1)m_d + \delta_0}\big) < \infty$, and if the $\pi$-centered function $\xi$ satisfies (\ref{xi-iterative}) and the following non-arithmeticity condition (a priori weaker for iterative models than the nonlattice assumption): \\[0.12cm]
%\noindent {\it Non-arithmeticity condition on $\cB_{\alpha,\beta,\gamma}$~: 
{\it There exist no $t\in\R^d$, $t\neq 0$, 
no $\lambda\in\C$, $|\lambda|=1$, no $\pi$-full $Q$-absorbing set $A\in\cE$, and finally no bounded function 
$w\in \cB_{\alpha,\vartheta,\gamma_{m_d}}$ whose modulus is nonzero constant on $A$, such that we have~: 
$\forall x\in A,\ \ e^{i\langle t,\xi(y)\rangle} w(y) = \lambda w(x)\ \ Q(x,dy)-\mbox{\it a.s}$.}
\end{rem}
%=================================================================================
%====================================================================================
\section{Conclusion} \label{conclusion}
In this paper we extend the well-known renewal theorem established by Spitzer \cite{spitzer} in dimension $d\geq3$ for centered nonlattice i.i.d.~sequences. In a first step, given a measurable space $E$ and a sequence $(X_n,S_n)_{n\geq0}$ of $E\times\R^d$-valued random variables, we present the ``tailor-made'' assumptions on $\E[f(X_n)\, e^{i\langle t, S_n\rangle}]$ under which such a renewal theorem can be proved by Fourier techniques. In a second step, considering the case when $(X_n)_{n\geq0}$ is a Markov chain, we give a general setting containing the Markov random walks, for which the spectral method may provide an efficient way to study the term $\E[f(X_n)\, e^{i\langle t, S_n\rangle}]$. This approach, already used in \cite{bab}, is improved here by appealing to the weak spectral method developed in \cite{fl}. Our main improvements concern the operator-type moment assumptions: when applied to the $v$-geometrically ergodic Markov chains, the $\rho$-mixing Markov chains and the iterative Lipschitz models, these assumptions are fulfilled 
under the (almost) expected moment condition in comparison with the i.i.d.~case. In particular, our results apply to unbounded r.v.~$S_1$, while the moment condition derived from \cite{bab} is not satisfied in general when $S_1$ is not bounded. \\
\indent The present weak spectral procedure should also enable to improve some of the results of \cite{bab,fuhlai}  concerning the non-centered Markov renewal theorem in dimension $d\geq2$. Work in these directions by the first author is in progress. Finally this method could be used to study the renewal theorem for Birkhoff sums in dynamical systems by using the so-called Perron-Frobenius operator, as already developed for other purposes in \cite{seb-08}. 
%=================================================================================
%====================================================================================
\newpage
%=====================================================================================
\begin{center}
\bf APPENDIX A. Complements in the proof of Theorem~\ref{ren-theo}
\end{center}
%\noindent{\bf APPENDIX A. Complements in the proof of Theorem~\ref{ren-theo}.} \\[0.5cm]
\noindent {\bf A.0. Inequalities for the derivatives of $1/w$.} \\[0.12cm]
\noindent Let $\gamma=(\gamma_1,\ldots,\gamma_d)\in\N^d$, and set $\vert\gamma\vert=\gamma_1+\ldots+\gamma_d$ and $\gamma!=\gamma_1!\ldots\gamma_d!$. 
If $\beta\in\N^d$ is such that $\beta\leq \gamma$, namely $\beta_i\leq\gamma_i$ for each $i\in\{1,\ldots,d\}$, we set ${\gamma\choose\beta} =\frac{\gamma!}{\beta!(\gamma-\beta)!}$. \\[0.12cm]
Let $\Omega$ be an open subset of $\R^d$, let $m\in\N^*$. We denote by $\cC^m(\Omega,\,\C)$ the space of  $m$-times continuously differentiable complex-valued functions on $\Omega$. If $\gamma\in\N^d$ is such that $\vert\gamma\vert\leq m$, we denote by  $\displaystyle{\partial^{\gamma}}$ the derivative operator defined on $\cC^m(\Omega,\,\C)$ by~:  
$$\partial^{\gamma} := \frac{\partial^{\vert\gamma\vert}}{\partial x_1^{\gamma_1}\ldots \partial x_d^{\gamma_d}} = \partial_1^{\gamma_1}\ldots\partial_d^{\gamma_d}\ \ \ \mbox{where}\ \ \ \partial_j := \frac{\partial}{\partial x_j}.$$
Let us recall (Leibniz's formula) that, if $f$ and $g$ are in $\cC^m(\Omega,\,\C)$, then we have for all $\gamma\in\N^d$ such that $\vert\gamma\vert\leq m$,
$$\partial^{\gamma}(f.\,g)=\sum_{\beta\leq\gamma}{\gamma \choose\beta}\partial^{\beta}f\cdot\,\partial^{\gamma-\beta}g.$$
Let $V$ be a bounded neighborhood of 0 in $\R^d$. 

\noindent {\bf Proposition A.0.} {\it Assume that $w : \overline{V}\r\C$ is $m$-times continuously differentiable, and that there exist some constants $c>0$ and $d\geq 0$ such that we have for all $x\in V$~: 
$\vert w(x)\vert \geq c\,\|x\|^2$ and $\sum_{j=1}^d\vert(\partial_j w)(x)\vert\leq d\,\|x\|$. Then, for each $\gamma\in \N^d$ such that $\vert \gamma\vert\leq m$, there exists a constant $C_{\gamma}\geq 0$ such that }
\begin{equation} \label{deri-1-over-w}
\forall x\in V\setminus\{0\},\ \ 
\big\vert\partial^{\gamma}(\frac{1}{w})(x)\big\vert\leq\frac{C_{\gamma}}{\|x\|^{2+\vert\gamma\vert}}.
\end{equation}
\noindent \begin{proof} 
Let $\gamma=(\gamma_1,\ldots,\gamma_d)\in\N^d$ be such that $\vert\gamma\vert\leq m$, set $M_{\gamma}=\max_{x\in \overline{V}}\vert\partial^{\gamma}w(x)\vert$, and let $K>0$ be such that $\|x\|\leq K$ for all $x\in V$. To prove (\ref{deri-1-over-w}), we are going to use an induction on $l=\vert\gamma\vert\in\{0,\ldots,m\}$. First, (\ref{deri-1-over-w}) is obvious if $l=0$, and if  $l=1$, then (\ref{deri-1-over-w}) follows from the following equalities 
$$\forall j\in\{1,\ldots,d\},\ \forall x\in V\setminus\{0\},\ \ \ \partial_j(\frac{1}{w})(x) = -\, \frac{\partial_j w(x)}{w^2(x)}.$$
Now let $l\in\{1,\ldots,m-1\}$, suppose that (\ref{deri-1-over-w}) is fulfilled for each $\gamma\in\N^d$ such that $\vert\gamma\vert\leq l$, and consider any $\gamma\in\N^d$ such that $\vert\gamma\vert=l+1$. Since $\partial^{\gamma}(w^{-1}\,.w)=0$, the Leibniz formula gives on $V\setminus\{0\}$:
$$\partial^{\gamma}(\frac{1}{w}) = -\frac{1}{w} \sum_{\beta\leq\gamma,\, \beta\neq\gamma} {\gamma \choose\beta}\partial^{\beta}(\frac{1}{w}).\,\partial^{\gamma-\beta}(w).$$
Let $\beta\leq \gamma$ be such that $\vert \gamma\vert-\vert\beta\vert\geq 2$ and let $x\in V\setminus\{0\}$. We have : $$\big\vert\frac{1}{w(x)}\, \partial^{\beta}(\frac{1}{w})(x)\, \partial^{\gamma-\beta}w(x)\big\vert\ \leq\ \frac{C_{\beta} M_{\gamma-\beta}}{c\,\|x\|^{4+\vert\beta\vert}}\ \leq\ 
\frac{K^{\vert\gamma\vert - \vert\beta\vert -2}C_{\beta} M_{\gamma-\beta}}{c\, \|x\|^{2+\vert\gamma\vert}}.$$
Let $\beta\leq \gamma$ be such that $\vert \gamma\vert-\vert\beta\vert=1$, and let $x\in V\setminus\{0\}$. Then we have : 
$$\big\vert\frac{1}{w(x)}\, \partial^{\beta}(\frac{1}{w})(x)\, \partial^{\gamma-\beta}w(x)\big\vert\ \leq\ \frac{d\,C_{\beta}}{c\,\|x\|^{3+\vert\beta\vert}}\ =\ \frac{d\,C_{\beta}}{c\,\|x\|^{2+\vert\gamma\vert}}.$$
This yields (\ref{deri-1-over-w}) for $\vert\partial^{\gamma}(\frac{1}{w})\vert$.  
\end{proof}
\noindent{\bf Remarks.} \\
{\bf (a)} The derivative estimates (\ref{deri-u}) needed for the function $u(t) = \frac{\theta_1(t)}{1-\lambda(t)}$ in the proof of Lemma~\ref{I1(a)-I3(a)} follows from this proposition. Indeed, let us set $w(t)=1-\lambda(t)$ for $t\in V:=\lbrack-\alpha,\alpha\rbrack^d$. Since $\nabla\lambda(0) = 0$ and $\Sigma$ is positive definite, $w$ satisfies the hypotheses of  Proposition A.0. Besides, one may assume that $\sqrt{d}\,\alpha<1$, so that each  $t\in V$ satisfies $\|t\| <1$. Now let $\gamma\in\N^d$ be such that $\vert\gamma\vert\leq d-2$, and let $C>0$ and $D_{\gamma}\geq 0$ be such that we have for all $t\in V$~: $\vert\theta_1(t)\vert\leq C\|t\|$ and $\displaystyle{\vert\partial^{\gamma}\theta_1(t)\vert\leq D_{\gamma}}$. On the one hand, we have 
$$\forall t\in V\setminus\{0\},\ \ \ \big\vert\theta_1(t)\, \partial^{\gamma}(\frac{1}{w})(t)\big\vert\leq \frac{C\, C_{\gamma}}{\|t\|^{1+\vert\gamma\vert}}.$$ 
On the other hand, for $\beta\in\N^d$ such that $\beta\not=\gamma$, $\beta\leq\gamma$, we have since $\vert \beta\vert\leq \vert\gamma\vert -1$ and $\|t\| <1$, 
$$\forall t\in V\setminus\{0\},\ \ \ \big\vert\partial^{\gamma-\beta}\theta_1(t)\, \partial^{\beta}(\frac{1}{w})(t)\big\vert\ \leq\  \frac{D_{\gamma-\beta}\, C_{\beta}}{\|t\|^{2+\vert\beta\vert}}\ \leq\ \frac{D_{\gamma-\beta}\, C_{\beta}}{\|t\|^{1+\vert\gamma\vert}}.$$
So (\ref{deri-u}) follows from Leibniz's formula. 

\noindent {\bf (b)} The derivative estimates (\ref{deri-v}) (\ref{deri-v'}) on the function $v(t) = \frac{\theta_3(t)}{(1-\lambda(t))\langle \Sigma t,t\rangle}$ in the proof of Lemma~\ref{I1(a)-I3(a)} can be derived similarly. Indeed, set $s(t)=\langle\Sigma t,t\rangle$ for $t\in\lbrack-\alpha,\alpha\rbrack^d$. Then the function $s(\cdot)$ satisfies the assumptions of Proposition A.0 on $V=\lbrack-\alpha,\alpha\rbrack^d$. Let $\gamma$ $\in\N^d$ be such that $\vert\gamma\vert\leq d-2$, and let $t\in V\setminus\{0\}$. From Leibniz's formula, the partial derivative $\displaystyle{\partial^{\gamma}v(t)}$ is a sum of terms of the form (up to some binomial coefficients)  
$$\partial^{\beta}\theta_3(t)\, \partial^{\delta}(\frac{1}{w})(t)\, \partial^{\eta}(\frac{1}{s})(t)\ \ \ \mbox{with}\ \  \vert\beta\vert+\vert\delta\vert + \vert\eta\vert=\vert\gamma\vert,$$
where $w(t)=1-\lambda(t)$. By proceeding as above (consider here the three cases  $\vert \beta\vert=\,0,1,2$, and the case $\vert \beta\vert \geq3$), one can then establish (\ref{deri-v}) and (\ref{deri-v'}). \\[0.5cm]
%===================================
\noindent {\bf A.1. Some inequalities for Fourier transforms.} 

\noindent Let $u : \R^d\setminus\{0\}\r\C$. We assume that there exists $\alpha>0$ such that we have $u(x) = 0$ for all $x\in\R^d\setminus[-\alpha,\alpha]^d$. 

\noindent{\bf Proposition A.1.} {\it Let $k\in\N^*$. If $u$ is $k$-times continuously differentiable on $\R^d\setminus\{0\}$ and satisfies the following conditions:  \\[0.15cm]
\noindent $(i)_k$  each partial derivative of order $\leq k-1$ of $u$ is $O(\|x\|^{-s})$ on $\R^d\setminus\{0\}\ $  for some 
$\ s<d-1$. \\[0.15cm]
%$\ \ \forall\ell=0,\ldots,k-1, \ \ |h^{(\ell)}(x)| = O(|x|^{-s})$ on $\R^d\setminus\{0\}\ $  for some $\ s<d-1$. \\[0.15cm]
\noindent $(ii)_k$ each partial derivative of order $k$ of $u$ is Lebesgue-integrable on $\R^d$, \\[0.15cm]
then we have $\displaystyle \hat u(a) = o\big(\frac{1}{\|a\|^k}\big)$ when $\|a\|\r+\infty$. }

\noindent This proposition is a consequence of the following lemma. 

\noindent{\bf Lemma A.1.} {\it If $u$ is continuously differentiable on $\R^d\setminus\{0\}$ and satisfies $(i)_1$-$(ii)_1$, then we have $\widehat{\frac{\partial u}{\partial x_j}}(a) = 
ia_j\, \hat u(a)$ for $j=1,\ldots,d$, and $\displaystyle \hat u(a) = o\big(\frac{1}{\|a\|}\big)$ when 
$\|a\|\r+\infty$. }

\noindent In fact, the second conclusion in Lemma A.1 proves Proposition A.1 in the case $k=1$, while the first conclusion of Lemma A.1 clearly allows to establish Proposition A.1 by induction. 

\noindent {\it Proof of Lemma A.1.} Below, $\langle \cdot,\cdot\rangle$ and $\|\cdot\|$ denote the scalar product and the euclidean norm in both $\R^d$ and $\R^{d-1}$. Let us write  $x=(x_1,x')\in\R\times\R^{d-1}$ and $a=(a_1,a')\in\R\times\R^{d-1}$, and define 
$$ u_1(x_1) = 
\int_{|x'|\leq \alpha}\, u(x_1,x') e^{-i\langle a',x'\rangle}\, dx',$$
so that we have 
$$\forall a\in\R^d,\ \ \ \hat u(a) = \int_{-\alpha}^\alpha u_1(x_1)\, e^{-ia_1x_1}\, dx_1.$$
Since $u$ is continuous on  $\R^d\setminus\{0\}$ and $x'\mapsto 1_{\{|x'|\leq \alpha\}}(x')\, \|x'\|^{-s}$ 
is  Lebesgue-integrable on $\R^{d-1}$ (because $s<d-1$), we deduce from  Condition $(i)_1$ and Lebesgue's theorem that 
$u_1$ is continuous on $\R$. 
Moreover, given any segment $K$ in $\R^*$, $\frac{\partial u}{\partial x_1}$ is bounded on the compact set 
$K\times [-\alpha,\alpha]^{d-1}$, and  it then follows from  Lebesgue's theorem that 
$u_1$ is continuously differentiable on $K$, with~: 
$$\forall x_1\in K,\ \ u_1'(x_1) = 
\int_{|x'|\leq \alpha}\, \frac{\partial u}{\partial x_1}(x_1,x') e^{-i\langle a',x'\rangle}\, dx'.$$ 
Now write 
$$\hat u(a) = \lim_{\varepsilon\r0}\bigg(\int_{\varepsilon}^\alpha u_1(x_1)\, e^{-ia_1x_1}\, dx_1 + 
\int_{-\alpha}^{-\varepsilon} u_1(x_1)\, e^{-ia_1x_1}\, dx_1\bigg).$$
For $a_1\neq0$, an easy integration by parts gives (notice that $u_1(\alpha)=0$ because $u(\alpha,\cdot)=0$) 
$$\int_{\varepsilon}^\alpha u_1(x_1)\, e^{-ia_1x_1}\, dx_1=  \frac{e^{-ia_1\varepsilon} u_1(\varepsilon)}{ia_1} + 
\frac{1}{ia_1}\int_{\varepsilon}^\alpha u_1'(x_1)\,  e^{-ia_1x_1}\, dx_1.$$
Using the continuity of $u_1(\cdot)$ at $0$, the above expression of $u_1'(x_1)$, and finally the fact that 
$\frac{\partial u}{\partial x_1}$ is Lebesgue-integrable on $\R^d$ by hypothesis, one gets 
$$\lim_{\varepsilon\r0}\bigg(\int_{\varepsilon}^\alpha u_1(x_1)\, e^{-ia_1x_1}\, dx_1\bigg) = 
\frac{u_1(0)}{ia_1} + \frac{1}{ia_1}\int_{[0,\alpha]\times [-\alpha,\alpha]^{d-1}}\,  \frac{\partial u}{\partial x_1}(x)  
e^{-i\langle a,x\rangle}\, dx.$$
Similarly one can prove that  
$$\lim_{\varepsilon\r0}\bigg(\int_{-\alpha}^{-\varepsilon} u_1(x_1)\, e^{-ia_1x_1}\, dx_1\bigg) = 
\frac{-u_1(0)}{ia_1} + \frac{1}{ia_1}\int_{[-\alpha,0]\times [-\alpha,\alpha]^{d-1}}\,  \frac{\partial u}{\partial x_1}(x)  
e^{-i\langle a,x\rangle}\, dx.$$
So $\hat u(a) =  \frac{1}{ia_1}\, \widehat{\frac{\partial u}{\partial x_1}}(a)$. On the same way, we have 
$\hat u(a) =  \frac{1}{ia_j}\, \widehat{\frac{\partial u}{\partial x_j}}(a)$ ($j=2,\ldots,d$). This proves the first assertion of 
Lemma A.1. This also gives $\|a\|\,|\hat u(a)| 
\leq \sum_{j=1}^d |a_j|\, |\hat u(a)| = \sum_{j=1}^d |\widehat{\frac{\partial u}{\partial x_j}}(a)|$, and since 
$\widehat{\frac{\partial u}{\partial x_j}}(a)\r0$ when $\|a\|\r+\infty$ (by $(ii)_1$), this yields the second assertion.  
\fdem
$\ $ \\
\noindent {\bf A.2. An elementary proof of (\ref{dist}).}\\[0.12cm]
\noindent  Let us recall that we set in (\ref{dist})~: $\ c=(2\pi)^{\frac{d}{2}}\, 2^{\frac{d}{2}-2}\, \Gamma(\frac{d-2}{2})$. Equality (\ref{dist}) follows from the following proposition. Let $S(\R^d)$ denote the Schwartz space.

\noindent {\bf Proposition A.2.} {\it Let $g\in S(\R^d)$. Then we have~: $\ \displaystyle \int_{\R^d} \hat{g}(u)\,\frac{1}{\|u\|^2}\,  du \, =\, c \int_{\R^d} g(v)\, \frac{1}{\|v\|^{d-2}}\,  dv$. } 

\noindent\begin{proof} First observe that 
$ \displaystyle\frac{1}{\|u\|^{2}} = \int_0^{+\infty}e^{-x\|u\|^2}\, dx$, 
so that we obtain by Fubini's theorem, Parseval's formula, and by setting $\gamma_x(u) = e^{-x\|u\|^2}$ 
$$\int_{\R^d} \hat{g}(u)\,\frac{1}{\|u\|^2}\,  du = 
\int_0^{+\infty} \int_{\R^d} \bigg(\hat{g}(u)\, e^{-x\|u\|^2}\, du\bigg)\, dx =  \int_0^{+\infty} \bigg(\int_{\R^d} g(v)\, 
\widehat{\gamma_x}(v)\, du\bigg)\, dx,$$
with $\widehat{\gamma_x}(v) = (\frac{\pi}{x})^{\frac{d}{2}}\,  e^{-\frac{\|v\|^2}{4x}}$ from a usual computation. 
Now Fubini's theorem again yields  
$$\int_0^{+\infty} \int_{\R^d} g(v)\, 
\widehat{\gamma_x}(v)\, du\, dx
= \pi^{\frac{d}{2}} \int_{\R^d} g(v)\, \bigg(\int_0^{+\infty} 
x^{-\frac{d}{2}}\,  e^{-\frac{\|v\|^2}{4x}}\, dx\bigg)\, dv.$$
Finally, we have $\int_0^{+\infty} x^{-\frac{d}{2}}\,  e^{-\frac{\|v\|^2}{4x}}\, dx = 2^{d-2}\, \Gamma(\frac{d-2}{2})\, \|v\|^{-(d-2)}$ (by setting $y=\frac{\|v\|^2}{4x}$). 
\end{proof}
$\ $ \\
\noindent {\bf A.3. Proof of (\ref{id-approchee}).}\\[0.12cm]
Recall that $*$ denotes the convolution product on $\R^d$, and that we set  $f_{d-2}(w) := \|w\|^{2-d}$ in 
(\ref{id-approchee}). Let us consider any $F\in S(\R^d)$, and set $F_\beta(x) = \beta^d\, F(\beta\, x)$ for any $\beta>0$. Then~: 

\noindent {\bf Proposition A.3.} {\it $\displaystyle\lim_{\beta\r+\infty}(F_\beta* f_{d-2})(\tilde b) = \int F(w)dw\ $ uniformly in $\tilde b$ such that $\|\tilde b\, \| = 1$. } 

\noindent\begin{proof} Let $\varepsilon>0$. Since 
$\|\tilde b\, \|=1$, there clearly exists $0<\eta=\eta(\varepsilon)<1$ independent of $\tilde b$ 
such that we have $\big|f_{d-2}(\tilde b - w) -1\big| \leq \varepsilon$ if $\|w\| < \eta$. From $\int F_\beta(w)dw = \int F(w)dw$, one gets 
\begin{eqnarray*}
\bigg |\int F(w)dw - (F_\beta* f_{d-2})(\tilde b)\bigg| &\leq&  \int |F_\beta(w)|\,  \big|1-f_{d-2}(\tilde b - w)\big|\, dw \\
&=& \int_{\|w\| < \eta} +  \int_{\|w\|\geq\eta} 
:= A_{\eta,\tilde b} + B_{\eta,\tilde b}.  
\end{eqnarray*}
We have $A_{\eta,\tilde b}\leq\varepsilon\int |F(w)|dw$. Besides we have 
$|F(\cdot)| \leq \frac{D}{(1 + \|\, \cdot\, \|)^{d+1}}$ 
for some $D>0$ (because $F\in S(\R^d)$), thus 
$B_{\eta,\tilde b} \leq B'_{\eta,\tilde b} + B''_{\eta,\tilde b}$ with $B'_{\eta,\tilde b} = \int_{\|w\|\geq\eta} |F_\beta(w)|\, dw$ and 
$$B''_{\eta,\tilde b} = 
\int_{\|w\|\geq\eta}  \bigg(\frac{D\, \beta^d}{\beta^{d+1}\|w\|^{d+1}}\bigg)\, \bigg(\frac{1}{\|\tilde b - w\|^{d-2}}\bigg)dw = 
\frac{D}{\beta}\, \int_{\|w\|\geq\eta}\frac{dw}{\|w\|^{d+1}\, \|\tilde b - w\|^{d-2}}.$$
For $\beta$ large enough, we have $B'_{\eta,\tilde b} =  \int_{\|y\|\geq \eta \beta} |F(y)| dy \leq \varepsilon$, and by decomposing 
$\int_{\|w\|\geq\eta}$ according that $\|\tilde b - w\|\leq2$ or $\|\tilde b - w\|>2$, and observing that 
$\|\tilde b - w\|>2\ \Rightarrow\ \|w\|>1$, we obtain 
$$B''_{\eta,\tilde b} \leq \frac{D}{\beta}\bigg(\frac{1}{\eta^{d+1}} \int_{\{\|\tilde b - w\|\leq2\}} 
\frac{dw}{\|\tilde b - w\|^{d-2}} +  
\frac{1}{2^{d-2}}\int_{\{\|w\|>1\}} \frac{dw}{\|w\|^{d+1}}\bigg).$$
Since the two previous integrals are finite, we have $B''_{\eta,\tilde b} \leq \varepsilon$ for $\beta$ large enough. 
\end{proof}


\newpage

\newpage


\begin{thebibliography}{99}

%\bibitem{als}
%{\sc Alsmeyer G.}
%{\em On the Markov renewal theorem.} 
%Stoch. Proc. Appl. (1994) 50, 37-56. (1994)

\bibitem{bab}
{\sc Babillot M.}
{\em Th\'eorie du renouvellement pour des  cha\^{\i}nes  semi-markoviennes transientes.}  
Ann. I. H. Poincar\'e, sect. B, Tome 24, No 4, 507-569 (1988). 

%\bibitem{bal}
%{\sc Baladi V.}
%{\em Positive transfer operators and decay of correlations.}  
%Advanced Series in Nonlinear Dynamics 16, World Scientific (2000). 

\bibitem{bal-seb}
{\sc B\'alint P., Gou\"ezel S.}
{\em Limit theorems in the stadium billiard.} 
Communications in Mathematical Physics {\bf 263}, 451-512 (2006). 

\bibitem{bre}
{\sc Breiman L.}
{\em Probability} Classic in Applied Mathematics, SIAM, 1993. 

\bibitem{broi}
{\sc Broise A., Dal'bo F., Peign\'e M.}
{\em \'Etudes spectrales d'op\'erateurs de transfert et applications.} Ast\'erisque 238 (1996).

%\bibitem{chazottes-gouezel}
%{\sc  Chazottes J.-R., Gou\"ezel S.} 
%{\em On almost-sure versions of classical limit theorems for dynamical systems.} 
%Probability Theory and Related Fields 138:195-234, 2007. 

\bibitem{BDG} 
{\sc D.~Buraczewski, E.~Damek, Y.~Guivarc'h.} 
{\em Convergence to stable laws for a class of multidimensional stochastic recursions.} 
Accepted for the publication in Probability Theory and Related Fields (2009). 

\bibitem{camp}
{\sc Campanato S.} 
{\em Proprieta di una famiglia di spazi funzionali.} Ann. Scuola Norm. Sup. Pisa (3) 18 (1964), 137-160. 


\bibitem{duf}
{\sc Duflo M.}
{\em Random Iterative Models.} 
Applications of Mathematics, Springer-Verlag Berlin Heidelberg (1997). 

\bibitem{fuhlai}
{\sc Fuh C.D, Lai T.L.}
{\em Asymptotic expansions in multidimensional Markov renewal theory and first passage times for Markov random walks.}
Adv. in Appl. Probab. {\bf 33}, 652-673 (2001). 

\bibitem{seb-08}
{\sc Gou\"ezel S.} 
{\em Necessary and sufficient conditions for limit theorems in Gibbs-Markov maps. }
Preprint (2008). 

\bibitem{gouliv}
{\sc Gou\"ezel S., Liverani C.}
{\em Banach spaces adapted to Anosov systems. } 
Ergodic Theory Dyn.~Syst. {\bf 26}, 189-217 (2006). 

%\bibitem{gouliv}
%{\sc Gou\"ezel S., Liverani C.}
%{\em Banach spaces adapted to Anosov systems. } 
%Ergodic Theory Dyn.~Syst. {\bf 26}, 189-217 (2006). 

\bibitem{denis}
{\sc Guibourg D.}
{\em Th\'eor\`eme de renouvellement pour  cha\^{\i}nes de Markov fortement ergodiques. 
Applications aux mod\`eles it\'eratifs Lipschitziens.} 
C. R. Acad. Sci. Paris, Ser. I 346 (2008) 435-438.  

\bibitem{guihar}
{\sc Guivarc'h Y., Hardy J.}
{\em Th\'eor\`emes limites pour une classe de cha\^{\i}nes de Markov et applications aux diff\'eomorphismes d'Anosov.} 
Ann. Inst. Henri Poincar\'e, Vol. 24, No 1, p. 73-98 (1988).

\bibitem{gui-lepage}
{\sc Guivarc'h Y., Le Page E.}
{\em On spectral properties of a family of transfer operators and convergence to stable laws for affine random walks.}
Ergodic Theory Dynam. Systems, 28 (2008) no.~2, pp.~423-446. 

%\bibitem{hen1}
%{\sc Hennion H.}  
%{\em  D\'erivabilit\'e du plus grand exposant caract\'eristique des produits de matrices al\'eatoires 
%ind\'ependantes \`a coefficients positifs. } 
%Ann. Inst. Henri Poincar\'e, Vol. 27, No 1 (1991) p. 27-59. 

%\bibitem{hen4}
%{\sc Hennion H.}
%{\em Quasi-compactness and absolutely continuous kernels.} 
%Probab. Theory Related Fields.
%{\bf 139} (2007) pp. 451-471. 

\bibitem{hulo}
{\sc Hennion H., Herv\'e L.}
{\em Limit theorems for Markov chains and stochastic properties of dynamical systems by quasi-compactness.} 
Lecture Notes in Mathematics No 1766, Springer (2001). 

\bibitem{aap}
{\sc Hennion H., Herv\'e L.}
{\em Central limit theorems for iterated random lipschitz mappings.} 
Annals of Proba. Vol. 32 No. 3A (2004) 1934-1984.  

\bibitem{ihp1}
{\sc Herv\'e L.}
{\em Th\'eor\`eme local pour cha\^{\i}nes de Markov de probabilit\'e de transition quasi-compacte. Applications aux 
cha\^{\i}nes $V$-g\'eom\'etriquement ergodiques et aux mod\`eles it\'eratifs.} 
Ann. I. H. Poincar\'e - PR 41 (2005) 179-196. 

%\bibitem{ihp2}
%{\sc Herv\'e L.}
%{\em Vitesse de convergence dans le théorème limite central pour chaînes de Markov fortement ergodiques. } 
%Ann. Inst. H. Poincaré Probab. Statist. {\bf 44} No. 2 (2008) 280-292. 

\bibitem{fl}
{\sc Herv\'e L, P\`ene F.}
{\em The Nagaev method via the Keller-Liverani theorem. } 
To appear in Bull. Soc. Math. France. See arXiv:0901.4617.  

%\bibitem{jlv}  
%{\sc Hervé L., Ledoux J., Patilea V.} 
%{\em A Berry-Esseen theorem on $M$-estimators for geometrically ergodic Markov chains.} 
%Preprint (2009). 

\bibitem{keli}
{\sc Keller G., Liverani C.}
{\em Stability of the Spectrum for Transfer Operators.} 
Ann. Scuola Norm. Sup. Pisa. CI. Sci. (4) Vol. XXVIII (1999) 141-152. 

%\bibitem{lep89}   
%{\sc Le Page E.}
%{\em R\'egularit\'e du plus grand exposant caract\'eristique des produits de matrices al\'eatoires ind\'ependantes et applications. } 
%Ann. Inst. Henri Poincar\'e, Vol. 25, No 2 (1989) p. 109-142.

\bibitem{liverani}
{\sc Liverani C.} 
{\em Invariant measures and their properties. A functional analytic point of view.}
Dynamical Systems. Part II: Topological Geometrical and Ergodic Properties of Dynamics. Pubblicazioni della Classe di Scienze, Scuola Normale Superiore, Pisa. Centro di Ricerca Matematica "Ennio De Giorgi": Proceedings. Published by the Scuola Normale Superiore in Pisa (2004). 

\bibitem{mey}  
{\sc S.P. Meyn and R.L. Tweedie.}
{\em Markov chains and stochastic stability.} 
Springer Verlag, New York, Heidelberg, Berlin (1993). 

\bibitem{nag1}
{\sc Nagaev S.V.}
{\em Some limit theorems for stationary Markov chains.} 
Theory of probability and its applications 11 4 (1957) 378-406.  

%\bibitem{nag2}  
%{\sc Nagaev S.V.}
%{\em More exact statements of limit theorems for homogeneous Markov chains.} 
%Theory of probability and its applications 6 1 (1961) 62-81.

%\bibitem{ney-spitzer}
%{\sc P.~Ney and F.~Spitzer.}
%{\em The Martin boundary for random walks.} 
%T.A.M.S, vol. {\bf 121}, 1966.

\bibitem{rosen}
{\sc Rosenblatt M. }
{\em Markov processes. Structure and asymptotic behavior. } 
Springer-Verlag. New York (1971). 

\bibitem{spitzer}
{\sc Spitzer F. }
{\em Principles of random walks. } 
Van Nostrand, Princeton, 1964. 

%\bibitem{uchi}
%{\sc Uchiyama K. }
%{\em Asymptotic estimates of the Green functions and transition probabilities for Markov additive processes. } 
%Electronic journal of Probability, {\bf 12}, pp. 138-180 (2007).  


\end{thebibliography}
\end{document}